\documentclass[12pt]{amsart}

\newcommand{\x}{\boldsymbol{x}}
\newcommand{\zd}{\,\mathrm{d}}
\newcommand{\abs}[1]{\left|#1\right|}

\newcommand{\bra}[1]{\left(#1\right)}
\newcommand{\brab}[1]{\big(#1\big)}
\newcommand{\braB}[1]{\Big(#1\Big)}
\newcommand{\brabg}[1]{\bigg(#1\bigg)}

\newcommand{\kbrab}[1]{\big[#1\big]}
\newcommand{\kbraB}[1]{\Big[#1\Big]}
\newcommand{\kbrabg}[1]{\bigg[#1\bigg]}

\newcommand{\myinner}[1]{\left\langle#1\right\rangle}
\newcommand{\myinnerb}[1]{\big\langle#1\big\rangle}
\newcommand{\myinnerB}[1]{\Big\langle#1\Big\rangle}
\newcommand{\myinnerbg}[1]{\bigg\langle#1\bigg\rangle}

\newcommand{\mynorm}[1]{\left\|#1\right\|}
\newcommand{\mynormb}[1]{\big\|#1\big\|}
\newcommand{\mynormB}[1]{\Big\|#1\Big\|}

\usepackage{hyperref}

\usepackage{amssymb}
\usepackage{amsthm}
\usepackage{amsxtra}

\usepackage{fancyhdr}
\usepackage{nccmath}
\usepackage{graphicx}
\usepackage{color}
\usepackage{amsmath}
\usepackage{listings} 
\usepackage{float}
\usepackage{cases} 
\usepackage{threeparttable}
\usepackage{dcolumn}
\usepackage{multirow}
\usepackage{booktabs}

\usepackage{caption}
\usepackage{subcaption}

\usepackage[english]{babel}

\newtheorem{theorem}{Theorem}[section]

\newtheorem{proposition}{Proposition}[section]
\newtheorem{lemma}[proposition]{Lemma}

\newtheorem{example}[proposition]{Example}

\theoremstyle{remark}
\newtheorem{remark}[proposition]{Remark}

\numberwithin{equation}{section}

\ifx\pdfoutput\undefined
\DeclareGraphicsExtensions{.pstex, .eps}
\else
\ifx\pdfoutput\relax
\DeclareGraphicsExtensions{.pstex, .eps}
\else
\ifnum\pdfoutput>0	
\DeclareGraphicsExtensions{.pdf}
\else
\DeclareGraphicsExtensions{.pstex, .eps}
\fi
\fi
\fi

\title[]{High-Order and Energy-Stable Implicit-Explicit Relaxation Runge-Kutta Schemes for Gradient Flows}
\author[]{Yuxiu Cheng}
\author[]{Kun Wang}
\author[]{Kai Yang}
\thanks{College of Mathematics and Statistics, Chongqing University, Chongqing 401331, P.R. China}


\begin{document}
	
	\maketitle
	
	\begin{abstract}
		In this paper, we propose a class of high-order and energy-stable implicit-explicit relaxation Runge-Kutta (IMEX RRK) schemes for solving the phase-field gradient flow models. By incorporating the scalar auxiliary variable (SAV) method, the original equations are reformulated into equivalent forms, and the modified energy is introduced. Then, based on the reformulated equations, we propose a kind of IMEX RRK methods, which are rigorously proved to preserve the energy dissipation law and achieve high-order accuracy for both Allen-Cahn and Cahn-Hilliard equations. Numerical experiments are conducted to validate the theoretical results, including the accuracy of the approximate solution and the efficiency of the proposed scheme. Furthermore, the schemes are extended to multi-component gradient flows, with the vector-valued Allen-Cahn equations serving as a representative example.
	\end{abstract}
	
	\section{Introduction}
	Gradient flow models describe the evolutionary process of a system transitioning toward thermodynamic equilibrium, providing a theoretical framework for studying the system's dynamics and related processes. Consequently, these models are widely applied across numerous scientific and engineering fields to simulate various complex phenomena. We consider the free energy system as follows:
	\begin{equation}\label{free_energy}
		{E}[u] = \frac{1}{2} \myinner{u, \mathcal{L} u} + {E}_1[u],
	\end{equation}
	where $\myinner{\cdot ,\cdot}$ denotes the $L^2$ inner product, $\mathcal{L}$ is a linear operator and ${E}_1[u]$ is a nonlinear functional. Then, the general equation of the gradient flow associated with the free energy \eqref{free_energy} can be expressed as
	\begin{equation}\label{gradient_flow}
		\begin{cases}
			u_t = \mathcal{G} \mu, \\
			\mu = \dfrac{\delta {E}}{\delta u},
		\end{cases}
	\end{equation}
	where $\mathcal{G}$ is a non-positive operator, which plays a key role in determining how energy dissipates in the system.
	Here, we focus on two widely used forms of $\mathcal{G}$ in phase-field models that share the same free energy:
	\begin{equation}\label{ac_ch_energy}
		E[u] = \int_{\Omega} \frac{\epsilon^2}{2} \abs{\nabla u}^2+ F(u) \, \zd \x,
	\end{equation}
	namely the Allen-Cahn equation ($\mathcal{G}=-\mathcal{I}$) and the Cahn-Hilliard equation ($\mathcal{G}=\Delta$). These models satisfy the energy dissipation law, i.e., $E[u(t_{n+1})] \leq E[u(t_{n})]$ for $ t_n \leq t_{n+1} $.
	
	In numerical simulation, preserving the energy dissipation law in numerical schemes is crucial for long-time simulations. Refer to the up-to-date literature review on this subject, where various approaches have been developed to achieve energy stability in phase field models. To ensure the energy stability and reduce the computational cost, Eyre proposed the convex splitting method \cite{Eyre1998Unconditionally}, which decomposes the free energy functional into the convex part and the non-convex part, treating them implicitly and explicitly, respectively. Building on this foundational work, Shen et al. \cite{Shen2012Second} extended the approach by developing an energy-stable second-order convex splitting scheme for solving the gradient flow equation with Ehrlich-Schwoebel type energy. Moreover, Feng et al. \cite{Feng2013Stabilized} proposed the energy stability of the numerical method through stabilization technology, which combines Crank-Nicolson method and Adams-Bashforth method for explicit processing of nonlinear terms. For a detailed error analysis of the stabilized semi-implicit method, see \cite{Yang2009Error}. In a parallel development, Cox and Matthews \cite{Matthews2002Exponential} proposed the exponential time differencing (ETD) method, which employs exponential integration techniques for the linear part and explicit treatment for the nonlinear part, thereby ensuring stability and improving computational effciency. Further advancements were made by Kassam and Trefethen \cite{Kassam2005Fourth}, who optimized the ETD method, significantly enhancing its performance in phase field models. And Du et al. \cite{Du2019Maximum} studied an energy-stable ETD scheme that preserves the maximum principle for Allen-Cahn equation. Another significant contribution came from Yang \cite{Yang2016Linear}, who proposed the invariant energy quadratization (IEQ) method. This approach quadratizes the free energy functional through auxiliary variables, enabling the design of an unconditionally energy-stable numerical scheme. Expanding on this research, Shen et al. \cite{Shen2018Convergence,Shen2018scalar} later introduced the scalar auxiliary variable (SAV) method, which simplifies the implementation of the IEQ method by using scalar auxiliary variables while preserving unconditional energy stability. Among these, the SAV approach has gained popularity due to its simplicity and ability to maintain unconditional energy stability. Further details, please refer to \cite{Badia2011Finite,Shen2017Numerical} for studies on coupling models, \cite{Li2019Energy,Li2019Convergence,Wang2011energy} for convergence analysis, and \cite{Guill2013linear,Li2016Characterizing,ShenYang2010Numerical,Wang2018efficient,Yang2017Numerical} for additional related works.

	Besides the energy stability, achieving high-order accuracy is equally important for effcient and accurate simulations. Traditional first-order methods often suffer from excessive numerical dissipation, limiting their applicability. To address this problem, high order methods such as Runge–Kutta (RK) methods \cite{Fu2024Higher,Fu2022Energy,Tang2022Arbitrarily,Yang2022Arbitrarily} and backward differentiation formula (BDF) methods \cite{Cheng2019BDF,Hao2021third,ZhangQiao2012adaptive} have been studied to improve the numerical accuracy. Especially, the RK methods achieve arbitrary high-order accuracy through multi-stage calculations and weighted averages. However the traditional explicit RK methods exhibit poor stability when applied to stiff systems, whereas the implicit RK methods and the BDF methods require higher computational costs from solving the nonlinear systems.
	To balance the accuracy and the effciency, the implicit-explicit (IMEX) method \cite{Ascher1997Implicit,Fu2024Energy,Li2017second} was considered, in which the equation was divided into stiff and non-stiff parts, treating them implicitly and explicitly. Building on this approach, Song \cite{Song2016SSP} proposed a Runge-Kutta method that combines SSP and IMEX techniques to solve the Cahn-Hilliard equation, and proved the unconditional energy stability of the lower-order methods. More recently, Kim et al. \cite{Kim2023Domains} proposed a linearized and mass-conservative IMEX RK finite difference scheme for solving the Cahn-Hilliard equation in arbitrary geometric domains, enabling the simulation of complex spatial geometries. However, higher-order time discretization methods may complicate energy-stable conditions ($E[u(t_{n+1})] \leq E[u(t_{n})]$) in gradient flow problems, making it challenging to combine high accuracy with energy stability. Recently, the relaxation techniques \cite{LiLi2023Relaxation,LiLi2024Multiple,LiLiZhang2023Implicit,Li2023Linearly} have been developed to achieve both high-order and energy-stable. Among these methods, relaxation Runge–Kutta (RRK) methods have emerged as a powerful tool for preserving the energy dissipation law in high-order schemes. By introducing a relaxation parameter, the RRK methods adjust the time step to ensure energy stability without sacrificing accuracy. 
	
	Based on this inspiration, we extend this method to the phase-field gradient flow models and propose 
	a class of high-order and energy-stable IMEX RRK scheme with the SAV formulation.
	In particular, the diagonal IMEX RK method is used to obtain the numerical approximation of the internal stage and a relaxation coefficient is introduced in the step update to ensure the modified energy dissipation. 
	Here, we use the SAV method to stabilize nonlinear terms avoiding numerical instability caused by direct handling, employ the IMEX RK method to balance high-order accuracy with computational efficiency, and utilize the relaxation technique ensures the modified energy stability.

	The paper is organized as follows. In Section 2, we propose a class of IMEX RRK schemes based on equivalent SAV form and prove these schemes preserve the modified energy dissipation law. In Section 3, we show the order of convergence for the proposed methods by estimating the relaxation parameters. Then, we validate the theoretical analyses through a series of numerical experiments in Section 4. Additionally, we extend the IMEX RRK schemes to multi-component gradient flows and further validate the theoretical results using numerical examples of the vector-valued Allen-Cahn equation. Finally, Section 5 concludes the paper.

	\section{IMEX RRK-SAV schemes}
	
	\subsection{SAV formulation}
	First, we recall the SAV method. Define a scalar function $r(t)$ as
	\begin{equation}
		r(t) := \sqrt{\int_{\Omega} F(u(\x, t)) \, \zd \x + C_0}, \quad t \geq 0,
	\end{equation}
	where $C_0$ is a non-negative constant ensuring that $r(t)$ remains real. 
	Then, from \cite{Shen2018scalar}, the gradient flow \eqref{gradient_flow} and energy functional \eqref{ac_ch_energy} are rewritten as
	\begin{equation}\label{ac_ch_sav}
		\begin{cases}
			u_t = \mathcal{G} \brabg{-\epsilon^2 \Delta u + \dfrac{r}{\sqrt{\int_{\Omega} F(u)\, \zd \x + C_0}} F^{\prime}(u)}, \\
			r_t = \dfrac{1}{2\sqrt{\int_{\Omega} F(u)\, \zd \x + C_0}} \myinner{F^{\prime}(u), u_t},
		\end{cases}
	\end{equation}
	and
	\begin{equation}\label{energy_sav}
		E[u,r] = \frac{\epsilon^2}{2} \mynormb{\nabla u}^2 + r^2 - C_0,
	\end{equation}
	where $\mynormb{\cdot}$ denotes the corresponding norm in $L^2$ space. 
	We define \eqref{energy_sav} to be the modified energy.
	For simplicity, we consider the problem with periodic boundary conditions on $(\x,t) \in \Omega \times [0,T_0]$, where $\Omega$ is a bounded rectangular domain, $T_0$ is a finite time, and the initial data is $u(\x,0) = u_0(\x)$.

	Without sacrificing generality, we will discuss the following simplified form:
	\begin{equation}\label{equation}
		\begin{cases}
			u_t = L(u) + N(u, r), \\
			r_t = \tilde{N}(u, r),
		\end{cases}
	\end{equation}
	where $L(u) = \mathcal{G} \bra{-\epsilon^2 \Delta u}$ represents the linear part, $N(u,r) = \mathcal{G} \braB{\frac{r}{\sqrt{\int_{\Omega} F(u)\, \zd \x + C_0}} F^{\prime}(u)}$ and $\tilde N(u,r) = \frac{1}{2\sqrt{\int_{\Omega} F(u)\, \zd \x + C_0}} \myinnerb{F^{\prime}(u), u_t}$ represent the nonlinear parts. 

	Before building the format, we begin by reviewing diagonally implicit-explicit Runge-Kutta (IMEX RK) methods with the SAV formulation and analyzing the stability properties of the modified energy.
	Then, we discuss the construction of diagonally implicit-explicit relaxation Runge-Kutta schemes 
	and investigate their energy dissipation law.
	
	\subsection{Diagonally IMEX RK schemes}
	Let $0 = t_0 < t_1 < \cdots < t_{M-1} < t_M = T_0$ be a partition of the interval $[0,T_0]$, where $M$ is a positive integer. 
	Consider a step within
	$[t_n,t_{n+1}] (0 \le n \le M-1)$ with stepsize $\tau$, and let $u^n$, $r^n$, $u^{n+1}$, $r^{n+1}$ 
	represent the numerical approximations of $u(t_n)$, $r(t_n)$, $u(t_{n+1})$, $r(t_{n+1})$, respectively. 
	The $s$-stage diagonally IMEX RK method for \eqref{equation} is then:
	
	\begin{equation}\label{imex rk-sav}
		\begin{cases}
			U_{ni} = u^n + \tau \sum\limits_{j=1}^i a_{ij}L_{nj} + \tau \sum\limits_{j=1}^{i-1} \bar a_{ij} N_{nj},\\
			R_{ni} = r^n + \tau \sum\limits_{j=1}^{i-1} \bar a_{ij}\tilde{N}_{nj}, \quad i=1,\cdots,s,\\
			u^{n+1} = u^n + \tau \sum\limits_{i=1}^s b_i L_{ni} + \tau \sum\limits_{i=1}^s \bar b_i N_{ni},\\
			r^{n+1} = r^n + \tau \sum\limits_{i=1}^s \bar b_i\tilde{N}_{ni},
		\end{cases}
	\end{equation}
	where $L_{nj}=L(U_{nj})$,\, $N_{nj}=N(U_{nj},R_{nj})$ and $\tilde{N}_{nj}=\tilde{N}(U_{nj},R_{nj})$,\, $j=1,\cdots,s$. With the following notations
	\begin{align*}
		&A = (a_{ij})_{s \times s}, \quad a_{ij} = 0 \quad \text{for} \quad j > i, \\
		&\bar{A} = (\bar{a}_{ij})_{s \times s}, \quad \bar{a}_{ij} = 0 \quad \text{for} \quad j \geq i, \\
		&b = (b_1, \dots, b_s)^T, \quad \bar{b} = (\bar{b}_1, \dots, \bar{b}_s)^T, \\
		&c = (c_1, \dots, c_s)^T, \quad \bar{c} = (\bar{c}_1, \dots, \bar{c}_s)^T.
	\end{align*}
	A double Butcher tableau can be used to describe the $s$-stage IMEX RK method \eqref{imex rk-sav} as follows,
	\[
	\begin{array}
		{c|c}
		c&
		A\\
		\hline
		& b^T
	\end{array}
	\qquad 
	\begin{array}
		{c|c}
		\bar c&
		\bar A\\
		\hline
		& \bar b^T
	\end{array},
	\]
	where $c_i=\sum_{j=1}^s a_{ij}$, $\bar c_i=\sum_{j=1}^s \bar a_{ij}$ and $\sum_{i=1}^s b_{i} = \sum_{i=1}^s \bar b_{i}=1$. 
	The choices of $a_{ij}$, $b_i$, and $c_i$ can be found in \cite{LiLiZhang2023Implicit}, and examples of these parameter choices are given in the Appendix.
	
	By substituting the third and fourth formulas in \eqref{imex rk-sav} into the modified energy \eqref{energy_sav}, the discrete energy at $t_{n+1}$ can be obtained
	\allowdisplaybreaks
	\begin{align*}
		& E^{n+1}[u^{n+1},r^{n+1}]=-\frac{\epsilon^2}{2} \myinner{\Delta u^{n+1}, u^{n+1}} + \bra{r^{n+1} }^2 -C_0 \\
		=&-\frac{\epsilon^2}{2} \myinnerB{\Delta \braB{u^n + \tau \sum\limits_{i=1}^s b_i L_{ni} + \tau \sum\limits_{i=1}^s \bar b_i N_{ni}} , u^n + \tau \sum\limits_{i=1}^s b_i L_{ni} + \tau \sum\limits_{i=1}^s \bar b_i N_{ni}} \\
		&+ \braB{r^n + \tau \sum\limits_{i=1}^s \bar b_i\tilde{N}_{ni}}^2 -C_0 \\
		=&E^n[u^{n},r^{n}] - \epsilon^2 \tau \sum\limits_{i=1}^s \myinnerB{U_{ni}, \Delta \bra{b_i L_{ni}+ \bar b_i N_{ni}}} +2R_{ni} \tau \sum\limits_{i=1}^s \bar b_i\tilde{N}_{ni}\\
		&- \tau \sum\limits_{i=1}^s \kbrabg{\epsilon^2 \myinnerB{u^n-U_{ni}, \Delta \bra{b_i L_{ni}+ \bar b_i N_{ni}}} -2 \bra{r^n-R_{ni}} \bar b_i\tilde{N}_{ni}} \\
		&-\tau^2 \kbrabg{\frac{\epsilon^2}{2} \myinnerB{\Delta \braB{\sum\limits_{i=1}^s  b_i L_{ni}+ \sum\limits_{i=1}^s \bar b_i N_{ni}}, \sum\limits_{i=1}^s  b_i L_{ni}+ \sum\limits_{i=1}^s \bar b_i N_{ni}} - \braB{\sum\limits_{i=1}^s \bar b_i\tilde{N}_{ni}}^2}.
	\end{align*}
	Using the first and second relations in \eqref{imex rk-sav}, i.e., $U_{ni} - u^n = \tau \sum_{j=1}^i a_{ij}L_{nj} + \tau \sum_{j=1}^{i-1} \bar a_{ij}N_{nj}$ and $R_{ni} - r^n = \tau \sum_{j=1}^{i-1} \bar a_{ij}\tilde{N}_{nj}$, we can further get from the above equation that
	\allowdisplaybreaks
	\begin{align*}
		& E^{n+1}[u^{n+1},r^{n+1}] = E^n[u^{n},r^{n}] - \epsilon^2 \tau \sum\limits_{i=1}^s \myinnerB{U_{ni}, \Delta \brab{b_i L_{ni}+ \bar b_i N_{ni}}} +2R_{ni} \tau \sum\limits_{i=1}^s \bar b_i\tilde{N}_{ni}\\
		& + \tau^2 \epsilon^2 \sum\limits_{i=1}^s \sum\limits_{j=1}^i \kbrabg{\myinnerB{a_{ij}L_{nj}, \Delta \brab{b_i L_{ni}+ \bar b_i N_{ni}}}} -\tau^2 \epsilon^2 \sum\limits_{i,j=1}^s \kbrabg{\myinnerB{\frac{1}{2}b_j L_{nj}, \Delta \brab{b_i L_{ni}+ \bar b_i N_{ni}}}}\\
		& + \tau^2 \epsilon^2 \sum\limits_{i=1}^s \sum\limits_{j=1}^{i-1} \kbrabg{\myinnerB{\bar a_{ij}N_{nj}, \Delta \brab{b_i L_{ni}+ \bar b_i N_{ni}}}} -\tau^2 \epsilon^2 \sum\limits_{i,j=1}^s \kbrabg{\myinnerB{\frac{1}{2} \bar b_j N_{nj}, \Delta \brab{b_i L_{ni}+ \bar b_i N_{ni}}}}\\
		& - \tau^2 \sum\limits_{i=1}^s \sum\limits_{j=1}^{i-1} 2(\bar a_{ij} \tilde N_{nj})(\bar b_i\tilde{N}_{ni}) + \tau^2 \sum\limits_{i,j=1}^s (\bar b_j \tilde N_{nj})(\bar b_i\tilde{N}_{ni})\\
		=& E^n[u^{n},r^{n}] - \epsilon^2 \tau \sum\limits_{i=1}^s \myinnerB{U_{ni}, \Delta \brab{b_i L_{ni}+ \bar b_i N_{ni}}} +2R_{ni} \tau \sum\limits_{i=1}^s \bar b_i\tilde{N}_{ni}\\
		& + \frac{1}{2} \tau^2 \epsilon^2 \sum\limits_{i,j=1}^{2s} \kbrab{\bra{2a_{ij}b_i-b_jb_i}\myinner{L_{nj}, \Delta L_{ni}}} + \frac{1}{2} \tau^2 \epsilon^2 \sum\limits_{i,j=1}^{2s} \kbrab{\bra{2a_{ij} \bar b_i-b_j \bar b_i}\myinner{L_{nj}, \Delta N_{ni}}}\\
		& + \frac{1}{2} \tau^2 \epsilon^2 \sum\limits_{i,j=1}^{2s} \kbrab{\bra{2\bar a_{ij}b_i-\bar b_jb_i}\myinner{N_{nj}, \Delta L_{ni}}} + \frac{1}{2} \tau^2 \epsilon^2 \sum\limits_{i,j=1}^{2s} \kbrab{\bra{2\bar a_{ij}\bar b_i-\bar b_j\bar b_i}\myinner{N_{nj}, \Delta N_{ni}}}\\
		& - \tau^2 \sum\limits_{i,j=1}^s \kbrab{\bra{2\bar a_{ij}\bar b_i-\bar b_j\bar b_i} \bra{\tilde N_{nj} \tilde N_{ni}}},
	\end{align*}
	hence,
	\begin{align}
		\begin{split}\label{Enms}
			E^{n+1}[u^{n+1},r^{n+1}] :=& E^n[u^{n},r^{n}] - \epsilon^2 \tau \sum\limits_{i=1}^s \myinnerB{U_{ni}, \Delta \brab{b_i L_{ni}+ \bar b_i N_{ni}}} +2R_{ni} \tau \sum\limits_{i=1}^s \bar b_i\tilde{N}_{ni} \\
			& - \frac{1}{2} \tau^2 \epsilon^2 \sum\limits_{i,j=1}^{2s} m_{ij} \myinnerb{\nabla Y_i, \nabla Y_j} - \tau^2 \sum\limits_{i,j=1}^{s} \tilde{s}_{ij}  \brab{\tilde{N}_{ni} \tilde{N}_{nj}}, 
		\end{split}
	\end{align}
	where
	\begin{align*}
		M = (m_{ij})_{2s \times 2s} = \begin{bmatrix}
			BA + A^T B & A^T \bar{B} + B \bar{A} \\
			\bar{B} A + \bar{A}^T B & \bar{B} \bar{A} + \bar{A}^T \bar{B}
		\end{bmatrix}
		- \begin{bmatrix}
			b b^T & b \bar{b}^T \\
			\bar{b} b^T & \bar{b} \bar{b}^T
		\end{bmatrix}, \,
		\tilde S = (\tilde{s}_{ij})_{s \times s} =\bar{B} \bar{A} + \bar{A}^T \bar{B} - \bar{b} \bar{b}^T,		
	\end{align*}
	and $B$ and $\bar B$ are diagonal matrices consisting of $b_j$ and $\bar b_j$, respectively,
	\begin{equation*}
		Y_j = \begin{cases}
			L_{nj}, & 1 \leq j \leq s, \\
			N_{n,j-s}, & s+1 \leq j \leq 2s.
		\end{cases}
	\end{equation*}

	From \cite{Li2019Energy,Tang2022Arbitrarily}, if $b_i=\bar b_i \geq 0$ for $i =1,\cdots,s$, we can similarly derive that the total of the second and third terms on the right-hand side of the relation \eqref{Enms} is non-positive, namely
	\begin{equation}\label{dissipative}
		- \epsilon^2 \tau \sum\limits_{i=1}^s \myinnerB{U_{ni}, \Delta \brab{b_i L_{ni}+ \bar b_i N_{ni}}} +2R_{ni} \tau \sum\limits_{i=1}^s \bar b_i\tilde{N}_{ni} \\
		=\, \tau \sum\limits_{i=1}^s b_i \myinnerb{\mu_{ni}, \mathcal{G} \mu_{ni}} \, \le 0,
	\end{equation}
	where $\mu_{ni}= -\epsilon^2 \Delta U_{ni} + \frac{R_{ni}}{\sqrt{\int_{\Omega} F(U_{ni})\, \zd \x + C_0}} F^{\prime}(U_{ni})$. 
	To ensure dissipation in the IMEX RK methods \eqref{imex rk-sav}, it is sufficient that the total of the fourth and fifth terms on the right-hand side of the relation \eqref{Enms} is non-positive.
	This condition can easily be achieved by using the semi-positive properties of matrices $M$ and $\tilde S$ (Refer to Definition 1 in \cite{Tang2022Arbitrarily}). 
	Moreover, the energy-dissipative property of the methods can be easily guaranteed if $M = 0$ and $S = 0$, implying that the symplectic condition must be satisfied by the implicit as well as explicit RK methods.
	However in the IMEX RK methods, satisfying the symplectic condition is quite challenging, primarily due to the explicit parts typically do not satisfy the symplectic condition.

	\subsection{IMEX RRK-SAV schemes} The explicit RK methods cannot achieve the algebraic or symplectic stability that implicit RK methods possess. 
	As a result, the IMEX RK methods \eqref{imex rk-sav} cannot guarantee the energy dissipation law.
	To address this limitation, we study the IMEX RRK methods by incorporating relaxation techniques in this subsection. Consider a single step over $[\hat{t}_n,\hat{t}_{n+1}]\,(0 \le n \le M-1)$ with stepsize $\tau$, and let $u_{\gamma}^n$,\, $r_{\gamma}^n$, \,$u_{\gamma}^{n+1}$,\, $r_{\gamma}^{n+1}$ 
	represent the numerical estimates of $u(\hat{t}_n)$, $r(\hat{t}_n)$, $u(\hat{t}_{n+1})$, $r(\hat{t}_{n+1})$, respectively. 
	We introduce a relaxation coefficient $\gamma_n$ to ensure the energy dissipation law when calculating the numerical solution $u_{\gamma}^{n+1}$ and $r_{\gamma}^{n+1}$. 
	Therefore, the $s$-stage diagonally IMEX RRK method for \eqref{equation} is written as:
	\begin{equation}\label{imex rrk-sav}
		\begin{cases}
			U_{ni} = u_{\gamma}^n + \tau \sum\limits_{j=1}^i a_{ij}L_{nj} + \tau \sum\limits_{j=1}^{i-1} \bar a_{ij}N_{nj},\\
			R_{ni} = r_{\gamma}^n + \tau \sum\limits_{j=1}^{i-1} \bar a_{ij}\tilde{N}_{nj}, \quad i=1,\cdots,s,\\
			u_{\gamma}^{n+1} = u_{\gamma}^n + \gamma_n\tau \sum\limits_{i=1}^s b_i L_{ni} + \gamma_n\tau \sum\limits_{i=1}^s \bar b_i N_{ni},\\
			r_{\gamma}^{n+1} = r_{\gamma}^n + \gamma_n\tau \sum\limits_{i=1}^s \bar b_i\tilde{N}_{ni},
		\end{cases}
	\end{equation}
	where $L_{nj}=L(U_{nj})$, $N_{nj}=N(U_{nj},R_{nj})$, $\tilde{N}_{nj}=\tilde{N}(U_{nj},R_{nj})$, $j=1,\cdots,s$. Clearly, the parameter $\gamma_n$ in \eqref{imex rrk-sav} is non-zero, otherwise $u_{\gamma}^{n+1} \not \equiv u_{\gamma}^n$. 
	Now, we show how to choose the relaxation coefficient $\gamma_n$ from analysis.
	
	Similarly, from \eqref{energy_sav} and \eqref{imex rrk-sav}, the discrete energy at $t_{n+1}$ is
	\begin{align}\label{En_gamma}
		\begin{split}
			&E^{n+1}[u_{\gamma}^{n+1},r_{\gamma}^{n+1}] = -\frac{\epsilon^2}{2} \myinner{\Delta u_{\gamma}^{n+1}, u_{\gamma}^{n+1}} + \bra{r_{\gamma}^{n+1} }^2 -C_0 \\
			=& E^n[u_{\gamma}^{n},r_{\gamma}^{n}] - \epsilon^2 \gamma_n \tau \sum\limits_{i=1}^s \myinnerb{U_{ni}, \Delta(b_i L_{ni}+ \bar b_i N_{ni})} +2R_{ni}\gamma_n \tau \sum\limits_{i=1}^s \bar b_i\tilde{N}_{ni}\\
			&-\gamma_n \tau \sum\limits_{i=1}^s \kbraB{\epsilon^2 \myinnerb{u_{\gamma}^n-U_{ni}, \Delta(b_i L_{ni}+ \bar b_i N_{ni})} -2(r_{\gamma}^n-R_{ni}) \bar b_i\tilde{N}_{ni}}\\
			&-\gamma_n^2 \tau^2 \kbraB{\frac{\epsilon^2}{2} \myinnerB{\Delta \braB{\sum\limits_{i=1}^s  b_i L_{ni}+ \sum\limits_{i=1}^s \bar b_i N_{ni}}, \sum\limits_{i=1}^s  b_i L_{ni}+ \sum\limits_{i=1}^s \bar b_i N_{ni}} - \braB{\sum\limits_{i=1}^s \bar b_i\tilde{N}_{ni}}^2}.
		\end{split}
	\end{align}

	Since $u_{\gamma}^{n+1}$ and $r_{\gamma}^{n+1}$ approximate the values of $u$ and $r$ at $\hat{t}_{n} + \gamma_n \tau$, respectively, it is crucial that $\gamma_n > 0$ must be satisfied at every step.
	For $\gamma_n > 0$ and the same parameters in \eqref{dissipative}, we also obtain the dissipation term is non-positive, namely
	\begin{align}\label{dissipation_gamma}
		\begin{split}
			&-\epsilon^2 \gamma_n \tau \sum\limits_{i=1}^s \myinnerb{U_{ni}, \Delta(b_i L_{ni}+ \bar b_i N_{ni})} +2R_{ni}\gamma_n \tau \sum\limits_{i=1}^s \bar b_i\tilde{N}_{ni}\\
			=& \gamma_n \tau \sum\limits_{i=1}^s b_i\kbrabg{-\epsilon^2 \myinnerb{U_{ni}, \Delta(L_{ni}+ N_{ni})} +2R_{ni}\tilde{N}_{ni}} \quad \le 0.
		\end{split}
	\end{align}
	
	Given that the total obtained by summing second and third terms on the right-hand side of the last relation in \eqref{En_gamma} is non-positive (see \eqref{dissipation_gamma}), we can choose the relaxation coeffcient $\gamma_n$ such that the sum of the last two terms vanishes to ensure the energy stability for the IMEX RRK schemes \eqref{imex rrk-sav}, namely
	\begin{multline*}
		-\gamma_n \tau \sum\limits_{i=1}^s \kbrabg{\epsilon^2 \myinnerB{u_{\gamma}^n-U_{ni}, \Delta\brab{b_i L_{ni}+ \bar b_i N_{ni}}} -2\brab{r_{\gamma}^n-R_{ni}} \bar b_i\tilde{N}_{ni}} \\
		-\gamma_n^2 \tau^2 \kbrabg{\frac{\epsilon^2}{2} \myinnerB{\Delta \braB{\sum\limits_{i=1}^s  b_i L_{ni}+ \sum\limits_{i=1}^s \bar b_i N_{ni}}, \sum\limits_{i=1}^s  b_i L_{ni}+ \sum\limits_{i=1}^s \bar b_i N_{ni}} - \braB{\sum\limits_{i=1}^s \bar b_i\tilde{N}_{ni}}^2} = 0,
	\end{multline*}
	by applying periodic /homogeneous dirichlet/neumann boundary conditions and integration by parts, one can easily see that
	\begin{multline*}
		-\gamma_n \tau \sum\limits_{i=1}^s \kbrabg{\epsilon^2 \myinnerB{u_{\gamma}^n-U_{ni}, \Delta\brab{b_i L_{ni}+ \bar b_i N_{ni}}} -2\brab{r_{\gamma}^n-R_{ni}} \bar b_i\tilde{N}_{ni}} \\
		+\gamma_n^2 \tau^2 \kbrabg{\frac{\epsilon^2}{2} \mynormB{\nabla \braB{\sum\limits_{i=1}^s  b_i L_{ni}+ \sum\limits_{i=1}^s \bar b_i N_{ni}}}^2 + \braB{\sum\limits_{i=1}^s \bar b_i\tilde{N}_{ni}}^2} = 0.
	\end{multline*}
	If 
	\begin{align}\label{denominator_ne}
		\frac{\epsilon^2}{2} \mynormB{\nabla \braB{\sum\limits_{i=1}^s  b_i L_{ni}+ \sum\limits_{i=1}^s \bar b_i N_{ni}}}^2 + \braB{\sum\limits_{i=1}^s \bar b_i\tilde{N}_{ni}}^2 \ne 0
	\end{align}
	holds, then
	\begin{align*}
		\gamma_n = \frac{\sum\limits_{i=1}^s \kbrabg{\epsilon^2 \myinnerB{u_{\gamma}^n-U_{ni}, \Delta\bra{b_i L_{ni}+ \bar b_i N_{ni}}} -2(r_{\gamma}^n-R_{ni}) \bar b_i\tilde{N}_{ni}}} {\tau \kbrabg{\frac{\epsilon^2}{2} \mynormB{\nabla \braB{\sum\limits_{i=1}^s  b_i L_{ni}+ \sum\limits_{i=1}^s \bar b_i N_{ni}}}^2 + \braB{\sum\limits_{i=1}^s \bar b_i\tilde{N}_{ni}}^2}}\,;
	\end{align*}
	otherwise, we directly set $\gamma_n=1$, representing the standard symplectic IMEX RK scheme, which, to our knowledge, we didn't know such choice of $a_{ij}$, $\bar a_{ij}$, $b_i$, $\bar b_i$,$c_i$ and $\bar c_i$. Therefore,
	\begin{align}\label{Def_gamma_n}
		\gamma_n =
		\begin{cases}
			\frac{\sum\limits_{i=1}^s \kbrabg{\epsilon^2 \myinnerB{u_{\gamma}^n-U_{ni}, \Delta\bra{b_i L_{ni}+ \bar b_i N_{ni}}} -2(r_{\gamma}^n-R_{ni}) \bar b_i\tilde{N}_{ni}}} {\tau \kbrabg{\frac{\epsilon^2}{2} \mynormB{\nabla \braB{\sum\limits_{i=1}^s  b_i L_{ni}+ \sum\limits_{i=1}^s \bar b_i N_{ni}}}^2 + \braB{\sum\limits_{i=1}^s \bar b_i\tilde{N}_{ni}}^2}}, \, &\text{if } \eqref{denominator_ne} \text{ holds},\\
			1, \, &\text{else}.
		\end{cases}
	\end{align}

	In the IMEX RRK schemes \eqref{imex rrk-sav}, there are two distinct interpretations of the numerical solutions $u_{\gamma}^{n+1}$ and $r_{\gamma}^{n+1}$ at $\hat{t}_{n+1}$ (cf.\cite{LiLiZhang2023Implicit}).
	\begin{itemize}
		\item \textbf{Incremental Direction Technique (IDT):} 
		$u_{\gamma}^{n+1}$ and $r_{\gamma}^{n+1}$ approximate the values of $u$ and $r$ at $\hat{t}_{n} + \tau$, where $\hat{t}_{n} = t_n$ for $n \geq 0$.
		
		\item \textbf{Relaxation Technique (RT):} 
		$u_{\gamma}^{n+1}$ and $r_{\gamma}^{n+1}$ approximate the values of $u$ and $r$ at $\hat{t}_{n} + \gamma_n \tau$, where $\hat{t}_{n} \ne t_n$ for $n > 0$.
		
	\end{itemize}
	
	It is important to note that varying interpretations of the numerical solutions $u_{\gamma}^{n+1}$ and $r_{\gamma}^{n+1}$ result in different convergence; further details can be found in Section 3. 
	Additionally, IMEX RRK$(s,p)$ denotes an $s$-stage, $p$-th order implicit-explicit relaxation RK method,
	with its coefficients listed in the appendix.
	
	\subsection{Energy dissipation law}
	Next, we discuss the modified energy's dissipation property of the IMEX RRK methods \eqref{imex rrk-sav} and give the following theorem.
	\begin{theorem}
		If $\gamma_n > 0$ and $b_j = \bar b_j  \ge 0$ for $j=1,\cdots,s$, then, IMEX RRK methods \eqref{imex rrk-sav} preserve the discrete energy decay property, namely
		\begin{align}\label{T:IMEXRRK_E}
			E^{n+1}[u_{\gamma}^{n+1},r_{\gamma}^{n+1}] \le E^n[u_{\gamma}^{n},r_{\gamma}^{n}].
		\end{align}
	\end{theorem}
	\begin{proof}
		If $\frac{\epsilon^2}{2} \mynormB{\nabla \braB{\sum\limits_{i=1}^s  b_i L_{ni}+ \sum\limits_{i=1}^s \bar b_i N_{ni}}}^2 + \braB{\sum\limits_{i=1}^s \bar b_i\tilde{N}_{ni}}^2=0$ holds, it is valid that
		\begin{align*}
			\begin{split}
				E^{n+1}[u_{\gamma}^{n+1},r_{\gamma}^{n+1}] &= E^n[u_{\gamma}^{n},r_{\gamma}^{n}] +\gamma_n \tau \epsilon^2 \myinnerbg{\nabla u_{\gamma}^n, \nabla \braB{\sum\limits_{i=1}^s b_i L_{ni}+ \sum\limits_{i=1}^s \bar b_i N_{ni}}}  + 2r_{\gamma}^n \gamma_n \tau \braB{\sum\limits_{i=1}^s \bar b_i\tilde{N}_{ni}}\\
				&\quad-\gamma_n^2 \tau^2 \kbraB{\frac{\epsilon^2}{2} \mynormB{\nabla \braB{\sum\limits_{i=1}^s  b_i L_{ni}+ \sum\limits_{i=1}^s \bar b_i N_{ni}}}^2 + \braB{\sum\limits_{i=1}^s \bar b_i\tilde{N}_{ni}}^2}\\
				&= E^n[u_{\gamma}^{n},r_{\gamma}^{n}],
			\end{split}
		\end{align*}
		that is, the energy dissipation law is automatically satisfied.
		
		In other cases, noting the assumptions that $\gamma_n > 0$ and $b_i = \bar{b}_i \geq 0$ for $i = 1, \cdots, s$, we have
		
		\begin{align*}
			E^{n+1}[u_{\gamma}^{n+1},r_{\gamma}^{n+1}] &= E^n[u_{\gamma}^{n},r_{\gamma}^{n}] - \epsilon^2 \gamma_n \tau \sum\limits_{i=1}^s \myinnerB{U_{ni}, \Delta \brab{b_i L_{ni}+ \bar b_i N_{ni}}} +2R_{ni}\gamma_n \tau \sum\limits_{i=1}^s \bar b_i\tilde{N}_{ni} \\
			&\le E^n[u_{\gamma}^{n},r_{\gamma}^{n}].
		\end{align*}
		The proof of \eqref{T:IMEXRRK_E} is completed due to the inequality \eqref{dissipation_gamma}.
		
	\end{proof}
	
	\section{Error analysis}
	This section is devoted to the error analysis of the IMEX RRK methods \eqref{imex rrk-sav}. 
	Our investigation starts with estimating the relaxation coeffcient $\gamma_n$, which is essential for the subsequent derivation.
	Based on this estimation, we then study the precision of our approximation solution.
	
	\subsection{Estimate of the relaxation coefficient}\label{Estimate:gamma} 
	Here, we focus on a step over $[\hat{t}_n,\hat{t}_{n+1}]$ with stepsize $\tau$. Note that $\dfrac{\epsilon^2}{2} \mynormB{\nabla \braB{\sum\limits_{i=1}^s  b_i L_{ni}+ \sum\limits_{i=1}^s \bar b_i N_{ni}}}^2 + \braB{\sum\limits_{i=1}^s \bar b_i\tilde{N}_{ni}}^2=0$ when $\gamma_n=1$, which corresponds to the standard IMEX RK method. Now we will concentrate on the case where the denominator is non-zero. Let 
	\begin{align}\label{Sn(gamma)}
		\begin{split}
			G_n(\gamma_n)=&E^{n+1}[u_{\gamma}^{n+1},r_{\gamma}^{n+1}]-E^n[u_{\gamma}^{n},r_{\gamma}^{n}]\\ 
			&- \kbrabg{-\epsilon^2 \gamma_n \tau \sum\limits_{i=1}^s \myinnerB{U_{ni}, \Delta \brab{b_i L_{ni}+ \bar b_i N_{ni}}} +2R_{ni}\gamma_n \tau \sum\limits_{i=1}^s \bar b_i\tilde{N}_{ni}}.
		\end{split}
	\end{align}
	Then, $\gamma_n$ in \eqref{Def_gamma_n} is the non-zero root of $G_n(\gamma)$. Regarding the function $G_n(\gamma)$, there holds the following Lemmas.

	\begin{lemma}\label{L:Sn(1)}
		Assume that the IMEX RK method \eqref{imex rk-sav} has $p$-th order accuracy, then
		$$G_n(1)= \mathcal{O}(\tau^{p+1}),$$	
		for small enough time step size $\tau$.
	\end{lemma}
	\begin{proof}
		Take into account the initial value problem as follows
		\begin{equation}
			\begin{cases}
				\tilde{u}_t = L(\tilde{u}) + N(\tilde{u},\tilde{r}), \\
				\tilde{r}_t = \tilde{N}(\tilde{u},\tilde{r}), \quad \hat{t}_n \le t \le \hat{t}_{n+1}, \\
				\tilde{u}(\hat{t}_n) = u^n_{\gamma},\\
				\tilde{r}(\hat{t}_n) = r^n_{\gamma},
			\end{cases}
		\end{equation}
		where $u^n_{\gamma}$ and $r^n_{\gamma}$ match those defined in \eqref{imex rrk-sav}. Executing one step using the $p$-th order IMEX RK method \eqref{imex rk-sav}, we obtain $\tilde{U}_{ni}$ for $i = 1, \cdots, s$ and $\tilde{u}^{n+1}$, $\tilde{r}^{n+1}$, satisfying 
		$$\tilde{u}^{n+1} = \tilde{u}(\hat{t}_n + \tau) + \mathcal{O}(\tau^{p+1}),\, \tilde{r}^{n+1} = \tilde{r}(\hat{t}_n + \tau) + \mathcal{O}(\tau^{p+1})$$
		for small enough time step size $\tau$. Starting from $u^n_\gamma$ and $r^n_\gamma$, $U_{ni}$ for $ i=1,\cdots,s$ and $u^{n+1}_\gamma$, $r^{n+1}_\gamma$ are calculated by using the IMEX RRK method \eqref{imex rrk-sav} with $\gamma_n=1$. 
		Noting $U_{ni}=\tilde U_{ni}$ for $i=1,\cdots,s$, $u_{\gamma}^{n+1} = \tilde{u}^{n+1}$ and $r_{\gamma}^{n+1} = \tilde{r}^{n+1}$, we substitute them into \eqref{energy_sav} to obtain $E^{n+1}[u_{\gamma}^{n+1},r_{\gamma}^{n+1}] = \tilde E^{n+1}[\tilde{u}^{n+1},\tilde{r}^{n+1}]$.
		
		Thus, denote that $\tilde{L}_{ni} = L(\tilde U_{ni})$,\,\,$\tilde{N}_{ni} = N(\tilde U_{ni},R_{ni})$ and $\tilde{\tilde{N}}_{ni} = \tilde{N}(\tilde U_{ni},R_{ni})$. From the fact that $E^{n+1}[u_{\gamma}^{n+1},r_{\gamma}^{n+1}] = \tilde E^{n+1}[\tilde{u}^{n+1},\tilde{r}^{n+1}]$ and $E^n[u_{\gamma}^{n},r_{\gamma}^{n}] = \tilde{E} \kbrab{\tilde{u}(\hat t_{n}),\tilde{r}(\hat t_{n})}$, by substituting $\gamma_n = 1$ into \eqref{Sn(gamma)}, we obtain
		\begin{align}
			\begin{split}\label{Sn(1)}
				G_n(1)&= E^{n+1}[u_{\gamma}^{n+1},r_{\gamma}^{n+1}] - E^n[u_{\gamma}^{n},r_{\gamma}^{n}] \\
				&\quad- \kbraB{-\epsilon^2 \tau \sum\limits_{i=1}^s \myinnerb{U_{ni}, \Delta(b_i L_{ni}+ \bar b_i N_{ni})} +2R_{ni} \tau \sum\limits_{i=1}^s \bar b_i\tilde{N}_{ni}}\\
				&=\tilde E^{n+1}[\tilde{u}^{n+1},\tilde{r}^{n+1}] - \tilde E \kbrab{\tilde{u}(\hat t_{n}),\tilde{r}(\hat t_{n})} \\
				&\quad - \kbraB{-\epsilon^2 \tau \sum\limits_{i=1}^s \myinnerb{\tilde U_{ni}, \Delta(b_i\tilde{L}_{ni}+ \bar b_i\tilde{N}_{ni}) } +2\tilde R_{ni} \tau \sum\limits_{i=1}^s \bar b_i \tilde{\tilde{N}}_{ni}}\\
				&=\tilde E^{n+1}[\tilde{u}^{n+1},\tilde{r}^{n+1}] - \tilde E\kbrab{\tilde{u}(\hat t_{n}+\tau),\tilde{r}(\hat t_{n}+\tau)} + \tilde E\kbrab{\tilde{u}(\hat t_{n}+\tau),\tilde{r}(\hat t_{n}+\tau)} \\
				&\quad- \tilde E \kbrab{\tilde{u}(\hat t_{n}),\tilde{r}(\hat t_{n})} - \kbraB{-\epsilon^2 \tau \sum\limits_{i=1}^s \myinnerb{\tilde U_{ni}, \Delta(b_i\tilde{L}_{ni}+ \bar b_i\tilde{N}_{ni}) } +2\tilde R_{ni} \tau \sum\limits_{i=1}^s \bar b_i \tilde{\tilde{N}}_{ni}}\\
				&=\mathcal{O}(\tau^{p+1}) - \kbraB{-\epsilon^2 \tau \sum\limits_{i=1}^s \myinnerb{\tilde U_{ni}, \Delta(b_i\tilde{L}_{ni}+ \bar b_i\tilde{N}_{ni}) } +2\tilde R_{ni} \tau \sum\limits_{i=1}^s \bar b_i \tilde{\tilde{N}}_{ni}}\\
				&\quad + \int_{\hat t_n}^{\hat t_{n}+\tau} \kbrabg{-\epsilon^2 \myinnerB{\tilde u(t), \Delta \brab{L(\tilde u(t))+ N(\tilde u(t),\tilde r(t))}} + 2\tilde r(t) \,\tilde{N}(\tilde u(t),\tilde r(t))} \,\zd t.
			\end{split}
		\end{align}
		Let $\Phi(t) = \int_{\hat t_n}^{t} \kbrabg{-\epsilon^2 \myinnerB{\tilde u(w), \Delta \brab{L(\tilde u(w))+ N(\tilde u(w),\tilde r(w))}} + 2\tilde r(w) \,\tilde{N}(\tilde u(w),\tilde r(w))} \,\zd w$ 
		and construct the system as follows:
		\begin{equation}\label{equation:phi+u}
			\begin{bmatrix}
				\Phi_t \\
				\tilde{u}_t\\
				\tilde{r}_t 
			\end{bmatrix}
			=
			\begin{bmatrix}
				-\epsilon^2 \myinnerb{\tilde{u}, \Delta L(\tilde{u})} \\
				L(\tilde{u})\\
				0
			\end{bmatrix}
			+
			\begin{bmatrix}
				-\epsilon^2 \myinnerb{\tilde{u}, \Delta N(\tilde{u},\tilde{r})} + 2 \tilde{r} \tilde{N}(\tilde{u},\tilde{r}) \\
				N(\tilde{u})\\
				\tilde{N}(\tilde{u},\tilde{r})
			\end{bmatrix},
			\quad \hat{t}_n \le t \le \hat{t}_{n+1},
		\end{equation}
		with initial value
		$$
		\begin{bmatrix}
			\Phi(\hat{t}_n) \\
			\tilde{u}(\hat{t}_n)\\
			\tilde{r}(\hat{t}_n)
		\end{bmatrix}
		=
		\begin{bmatrix}
			0 \\
			u_{\gamma}^{n}\\
			r_{\gamma}^{n}
		\end{bmatrix}.
		$$
		Using the $p$-th order IMEX RK method \eqref{imex rk-sav} to solve the system \eqref{equation:phi+u}, one obtains
		
		\begin{align*}
			&\Phi(\hat t_n+\tau) = \Phi^{n+1}+\mathcal{O}(\tau^{p+1}) \\
			=& \Phi^{n} + \tau \sum\limits_{i=1}^s b_i \kbrabg{-\epsilon^2 \myinnerB{\tilde U_{ni}, \Delta \tilde L_{ni})}} + \tau \sum\limits_{i=1}^s \bar b_i \kbrabg{-\epsilon^2 \myinnerB{\tilde U_{ni}, \Delta\tilde N_{ni}} + 2\tilde R_{ni} \tilde{\tilde{N}}_{ni}} +\mathcal{O}(\tau^{p+1}) \\
			=& -\epsilon^2 \tau \sum\limits_{i=1}^s \myinnerB{\tilde U_{ni}, \Delta(b_i \tilde L_{ni}+ \bar b_i \tilde N_{ni})} + 2\tilde R_{ni} \tau \sum\limits_{i=1}^s \bar b_i \tilde{\tilde{N}}_{ni}+\mathcal{O}(\tau^{p+1}),
		\end{align*}
		where $\Phi^{n+1}$ approximates the values of $\Phi$ at $\hat t_n+\tau$ and $\Phi^{n}=\Phi(\hat{t}_n)=0$. Therefore, the above equality demonstrates that
		\begin{align*}
			&\int_{\hat t_n}^{\hat t_{n}+\tau} \kbrabg{-\epsilon^2 \myinnerB{\tilde u(t), \Delta \brab{L(\tilde u(t))+ N(\tilde u(t),\tilde r(t))}} + 2\tilde r(t) \,\tilde{N}(\tilde u(t),\tilde r(t))} \,\zd t \\
			= &-\epsilon^2 \tau \sum\limits_{i=1}^s \myinnerB{\tilde U_{ni}, \Delta(b_i \tilde L_{ni}+ \bar b_i \tilde N_{ni})} + 2\tilde R_{ni} \tau \sum\limits_{i=1}^s \bar b_i \tilde{\tilde{N}}_{ni}+\mathcal{O}(\tau^{p+1}),
		\end{align*}
		substituting them into the last relation of \eqref{Sn(1)}, we obtain
		\begin{align*}
			G_n(1)&= \mathcal{O}(\tau^{p+1}) - \kbraB{-\epsilon^2 \tau \sum\limits_{i=1}^s \myinnerb{\tilde U_{ni}, \Delta(b_i\tilde{L}_{ni}+ \bar b_i\tilde{N}_{ni}) } +2\tilde R_{ni} \tau \sum\limits_{i=1}^s \bar b_i \tilde{\tilde{N}}_{ni}}\\
			&\quad+ \int_{\hat t_n}^{\hat t_{n}+\tau} \kbrabg{-\epsilon^2 \myinnerB{\tilde u(t), \Delta \brab{L(\tilde u(t))+ N(\tilde u(t),\tilde r(t))}} + 2\tilde r(t) \,\tilde{N}(\tilde u(t),\tilde r(t))} \,\zd t\\
			&= \mathcal{O}(\tau^{p+1}).
		\end{align*}
		This completes the proof.
	\end{proof}
	
	\begin{lemma}\label{L:S'n(0) and S'n(1)}
		Assume that the IMEX RK method \eqref{imex rk-sav} has $p$-th order accuracy ($p \geq 2$), then
		$$ G_n^{'}(0)= -g\tau^2 + \mathcal{O}(\tau^{3})  \quad and  \quad G_n^{'}(1)= g\tau^2 + \mathcal{O}(\tau^{3}),$$
		for small enough time step size $\tau$, where $g$ is a constant that may vary at different occurences.
	\end{lemma}
	\begin{proof}
		Noting the definition in \eqref{Sn(gamma)} and using the second relations in \eqref{imex rk-sav} yields
		\begin{align*}
			&G^{'}_n(\gamma_n)=\epsilon^2 \myinnerB{\nabla u^n_{\gamma} + \gamma_n \tau \sum\limits_{i=1}^s  \brab{b_i \nabla L_{ni}+\bar b_i \nabla N_{ni}},  \tau \sum\limits_{i=1}^s \brab{b_i \nabla L_{ni}+\bar b_i \nabla N_{ni}}} \\
			&\quad+2\gamma_n \tau^2 \braB{\sum\limits_{i=1}^s \bar b_i\tilde{N}_{ni}}^2 
			+2 r^n_{\gamma}\tau \sum\limits_{i=1}^s \bar b_i\tilde{N}_{ni} + \epsilon^2 \tau \sum\limits_{i=1}^s \myinnerB{U_{ni},\Delta \brab{b_i L_{ni}+ \bar b_i N_{ni}}} - 2R_{ni} \tau \sum\limits_{i=1}^s \bar b_i\tilde{N}_{ni}\\		
			&=\epsilon^2 \myinnerB{\nabla u^n_{\gamma} + \gamma_n \tau \sum\limits_{i=1}^s  \brab{b_i \nabla L_{ni}+\bar b_i \nabla N_{ni}},  \tau \sum\limits_{i=1}^s \brab{b_i \nabla L_{ni}+\bar b_i \nabla N_{ni}}} +2\gamma_n \tau^2 \braB{\sum\limits_{i=1}^s \bar b_i\tilde{N}_{ni}}^2 \\
			&\quad-2 \tau^2 \braB{\sum\limits_{i=1}^s \bar b_i\tilde{N}_{ni}} \braB{\sum\limits_{j=1}^{i-1} \bar a_{ij}\tilde{N}_{nj}} +\epsilon^2 \tau \sum\limits_{i=1}^s \myinnerB{U_{ni},\Delta \brab{b_i L_{ni}+ \bar b_i N_{ni}}}.
		\end{align*}
		Then, substituting $\gamma_n = 0$ and $\gamma_n = 1$ respectively and using the first relations in \eqref{imex rk-sav}, we obtain
		\begin{align*}
			G^{'}_n(0) &= \epsilon^2 \myinnerB{\nabla u^n_{\gamma},  \tau \sum\limits_{i=1}^s \brab{b_i \nabla L_{ni}+\bar b_i \nabla N_{ni}} } +\epsilon^2 \tau \sum\limits_{i=1}^s \myinnerB{U_{ni},\Delta \brab{b_i L_{ni}+ \bar b_i N_{ni}}}\\
			&\quad-2 \tau^2 \braB{\sum\limits_{i=1}^s \bar b_i\tilde{N}_{ni}} \braB{\sum\limits_{j=1}^{i-1} \bar a_{ij}\tilde{N}_{nj}}\\
			&=-\epsilon^2 \tau \sum\limits_{i=1}^s \myinnerB{u^n_{\gamma}-U_{ni}, \Delta \brab{b_i L_{ni}+\bar b_i N_{ni}}} - 2 \tau^2 \braB{\sum\limits_{i=1}^s \bar b_i\tilde{N}_{ni}} \braB{\sum\limits_{j=1}^{i-1} \bar a_{ij}\tilde{N}_{nj}}\\
			&=\epsilon^2 \tau^2 \sum\limits_{i=1}^s \myinnerB{\sum\limits_{j=1}^s \brab{a_{ij}L_{nj}+\bar{a}_{ij}N_{nj}}, \Delta \brab{b_i L_{ni}+\bar b_i N_{ni}}} -2 \tau^2 \braB{\sum\limits_{i=1}^s \bar b_i\tilde{N}_{ni}} \braB{\sum\limits_{j=1}^{i-1} \bar a_{ij}\tilde{N}_{nj}}\\
			&=\epsilon^2 \tau^2 \sum\limits_{i,j=1}^s a_{ij}b_i \myinnerb{L_{nj}, \Delta L_{ni}} + \epsilon^2 \tau^2 \sum\limits_{i,j=1}^s a_{ij} \bar b_i \myinnerb{L_{nj}, \Delta N_{ni}} +\epsilon^2 \tau^2 \sum\limits_{i,j=1}^s \bar a_{ij}b_i \myinnerb{N_{nj}, \Delta L_{ni}} \\
			&\quad + \epsilon^2 \tau^2 \sum\limits_{i,j=1}^s \bar a_{ij} \bar b_i \myinnerb{N_{nj}, \Delta N_{ni}} -2 \tau^2 \braB{\sum\limits_{i=1}^s \bar b_i\tilde{N}_{ni}} \braB{\sum\limits_{j=1}^{i-1} \bar a_{ij}\tilde{N}_{nj}},
		\end{align*}
		and
		\allowdisplaybreaks
		\begin{align*}
			G^{'}_n(1)&=\epsilon^2 \myinnerB{\nabla u^n_{\gamma} + \tau \sum\limits_{i=1}^s \brab{b_i \nabla L_{ni}+\bar b_i \nabla N_{ni}},  \tau \sum\limits_{i=1}^s \brab{b_i \nabla L_{ni}+\bar b_i \nabla N_{ni}} } \\
			&\quad +2 \tau^2 \braB{\sum\limits_{i=1}^s \bar b_i\tilde{N}_{ni}}^2 +\epsilon^2 \tau \sum\limits_{i=1}^s \myinnerB{U_{ni},\Delta \brab{b_i L_{ni}+ \bar b_i N_{ni}}} -2 \tau^2 \brab{\sum\limits_{i=1}^s \bar b_i\tilde{N}_{ni}} \brab{\sum\limits_{j=1}^{i-1} \bar a_{ij}\tilde{N}_{nj}} \\
			&=\epsilon^2 \tau \sum\limits_{i=1}^s \myinnerB{\nabla u^n_{\gamma} + \tau \sum\limits_{j=1}^s \brab{b_j \nabla L_{nj}+\bar b_j \nabla N_{nj}} , \brab{b_i \nabla L_{ni}+\bar b_i \nabla N_{ni}}} \\
			&\quad +2 \tau^2 \brab{\sum\limits_{i=1}^s \bar b_i\tilde{N}_{ni}}^2 +\epsilon^2 \tau \sum\limits_{i=1}^s \myinnerB{U_{ni},\Delta \brab{b_i L_{ni}+ \bar b_i N_{ni}}} -2 \tau^2 \brab{\sum\limits_{i=1}^s \bar b_i\tilde{N}_{ni}} \brab{\sum\limits_{j=1}^{i-1} \bar a_{ij}\tilde{N}_{nj}} \\
			&=-\epsilon^2 \tau \sum\limits_{i=1}^s \myinnerB{u^n_{\gamma}-U_{ni} + \tau \sum\limits_{j=1}^s  b_j L_{nj}+ \tau \sum\limits_{j=1}^s \bar b_j N_{nj}, \Delta \brab{b_i L_{ni}+\bar b_i N_{ni}}} \\
			&\quad +2 \tau^2 \brab{\sum\limits_{i=1}^s \bar b_i\tilde{N}_{ni}}^2 -2 \tau^2 \brab{\sum\limits_{i=1}^s \bar b_i\tilde{N}_{ni}} \brab{\sum\limits_{j=1}^{i-1} \bar a_{ij}\tilde{N}_{nj}}\\
			&=-\epsilon^2 \tau^2 \sum\limits_{i=1}^s \myinnerbg{\sum\limits_{j=1}^s \kbraB{\brab{b_j-a_{ij}} L_{nj}+ \brab{\bar b_j-\bar a_{ij}} N_{nj}}, \Delta \brab{b_if_{ni}+\bar b_i N_{ni}}} \\
			&\quad +2 \tau^2 \braB{\sum\limits_{i=1}^s \bar b_i\tilde{N}_{ni}} \kbraB{\sum\limits_{j=1}^s \bar b_j\tilde{N}_{nj}-\sum\limits_{j=1}^{i-1} \bar a_{ij}\tilde{N}_{nj}}\\
			&=-\epsilon^2 \tau^2 \sum\limits_{i,j=1}^s b_i (b_j-a_{ij}) \myinnerb{L_{nj}, \Delta L_{ni}} -\epsilon^2 \tau^2 \sum\limits_{i,j=1}^s b_i (\bar b_j-\bar a_{ij}) \myinnerb{N_{nj}, \Delta L_{ni}}\\
			&\quad - \epsilon^2 \tau^2 \sum\limits_{i,j=1}^s \bar b_i (b_j-a_{ij}) \myinnerb{L_{nj}, \Delta N_{ni}} -\epsilon^2 \tau^2 \sum\limits_{i,j=1}^s \bar b_i (\bar b_j-\bar a_{ij}) \myinnerb{N_{nj}, \Delta N_{ni}} \\
			&\quad +2 \tau^2 \braB{\sum\limits_{i=1}^s \bar b_i\tilde{N}_{ni}} \kbraB{\sum\limits_{j=1}^s (\bar b_j-\bar a_{ij})\tilde{N}_{nj}}.
		\end{align*}
		
		Using $\tilde L_n = L(u(\tilde t_n))$, $\tilde N_n = N(u(\tilde t_n),r(\tilde t_n))$, $\tilde{\tilde N}_n=\tilde N(u(\tilde t_n),r(\tilde t_n))$ and Taylor expansion, One can easily show that $L_{ni},L_{nj}=\tilde L_n+\mathcal{O}(\tau)$ and $N_{ni},N_{nj}=\tilde N_n+\mathcal{O}(\tau)$ and $\tilde N_{ni},\tilde N_{nj}=\tilde{\tilde N}_n+\mathcal{O}(\tau)$ for small enough time step size $\tau$. Hence,
		\begin{align*}
			G^{'}_n(0)&=\epsilon^2 \tau^2 \sum\limits_{i,j=1}^s a_{ij}b_i \myinnerb{\tilde L_{n}, \Delta \tilde L_{n}} + \epsilon^2 \tau^2 \sum\limits_{i,j=1}^s a_{ij} \bar b_i \myinnerb{\tilde L_{n}, \Delta \tilde{N}_n} +\epsilon^2 \tau^2 \sum\limits_{i,j=1}^s \bar a_{ij}b_i \myinnerb{\tilde N_{n}, \Delta \tilde L_{n}} \\
			&\quad + \epsilon^2 \tau^2 \sum\limits_{i,j=1}^s \bar a_{ij} \bar b_i \myinnerb{\tilde N_{n}, \Delta \tilde N_{n}} - 2 \tau^2 \sum\limits_{i,j=1}^s \bar a_{ij} \bar b_i \bra{\tilde{\tilde{N}}_{n}}^2 +\mathcal{O}(\tau^3),
		\end{align*}
		and
		\begin{align*}
			G^{'}_n(1)&=-\epsilon^2 \tau^2 \sum\limits_{i,j=1}^s b_i (b_j-a_{ij}) \myinnerb{\tilde L_{n}, \Delta \tilde L_{n}} -\epsilon^2 \tau^2 \sum\limits_{i,j=1}^s b_i (\bar b_j-\bar a_{ij}) \myinnerb{\tilde N_{n}, \Delta \tilde L_{n}} \\
			&\quad - \epsilon^2 \tau^2 \sum\limits_{i,j=1}^s \bar b_i (b_j-a_{ij}) \myinnerb{\tilde L_{n}, \Delta \tilde N_{n}} - \epsilon^2 \tau^2 \sum\limits_{i,j=1}^s \bar b_i (\bar b_j-\bar a_{ij}) \myinnerb{\tilde N_{n}, \Delta \tilde N_{n}} \\
			&\quad +2 \tau^2 \sum\limits_{i,j=1}^s \bar b_i (\bar b_j-\bar a_{ij}) \bra{\tilde{\tilde{N}}_{n}}^2 +\mathcal{O}(\tau^3).
		\end{align*}
		Applying the second-order conditions, being required for second order and above methods, for the IMEX RK methods \cite{Kennedy2003Additive}, 
		namely
		\begin{align}
			\sum\limits_{i,j=1}^s a_{ij}b_i =\frac{1}{2},\,\sum\limits_{i,j=1}^s \bar a_{ij}b_i =\frac{1}{2},\,
			\sum\limits_{i,j=1}^s a_{ij} \bar b_i=\frac{1}{2},\,\sum\limits_{i,j=1}^s \bar a_{ij} \bar b_i =\frac{1}{2},
		\end{align}
		and
		\begin{align}
			\sum\limits_{i,j=1}^s (b_j-a_{ij})b_i =\frac{1}{2},\, \sum\limits_{i,j=1}^s (\bar b_j-\bar a_{ij})b_i =\frac{1}{2},\,
			\sum\limits_{i,j=1}^s (b_j-a_{ij})\bar b_i=\frac{1}{2},\,\sum\limits_{i,j=1}^s (\bar b_j-\bar a_{ij}) \bar b_i=\frac{1}{2}.
		\end{align}
		Then we can derive
		\begin{align*}
			G^{'}_n(0)&=\frac{1}{2}\epsilon^2 \tau^2 \kbraB{\myinnerb{\tilde L_{n}, \Delta \tilde L_{n}} + \myinnerb{\tilde L_{n}, \Delta \tilde N_{n}} + \myinnerb{\tilde N_{n}, \Delta \tilde L_{n}} + \myinnerb{\tilde N_{n}, \Delta \tilde N_{n}}} - \tau^2 \bra{\tilde{\tilde{N}}_{n}}^2+\mathcal{O}(\tau^3)\\
			&= - \kbrabg{\frac{1}{2}\epsilon^2 \myinnerB{\nabla \brab{\tilde L_{n}+\tilde N_{n}}, \nabla \brab{\tilde L_{n}+\tilde N_{n}}} +(\tilde{\tilde{N}})^2} \tau^2 +\mathcal{O}(\tau^3),
		\end{align*}
		and
		\begin{align*}
			G^{'}_n(1)&=-\frac{1}{2} \epsilon^2 \tau^2 \kbraB{\myinnerb{\tilde L_{n}, \Delta \tilde L_{n}} + \myinnerb{\tilde L_{n}, \Delta \tilde N_{n}} + \myinnerb{\tilde N_{n}, \Delta \tilde L_{n}} + \myinnerb{\tilde N_{n}, \Delta \tilde N_{n}}} + \tau^2 \bra{\tilde{\tilde{N}}_{n}}^2 + \mathcal{O}(\tau^3)\\
			&= \kbrabg{\frac{1}{2}\epsilon^2 \myinnerB{\nabla \brab{\tilde L_{n}+\tilde N_{n}}, \nabla \brab{\tilde L_{n}+\tilde N_{n}}} +(\tilde{\tilde{N}})^2} \tau^2 +\mathcal{O}(\tau^3).
		\end{align*}
		This completes the proof.
	\end{proof}
	
	\begin{remark}	
		The condition that the step size needs to be small enough is because we apply the Taylor expansion in our analysis.
	\end{remark}
	
	Combining the Lemmas \ref{L:Sn(1)} and \ref{L:S'n(0) and S'n(1)}, we derive an estimate for 
	the relaxation coefficient $\gamma_n$ given in \eqref{imex rrk-sav}, as presented in Theorem \ref{T:IMEXRRK gamma}.
	\begin{theorem}\label{T:IMEXRRK gamma}
		Assume that the IMEX RK method \eqref{imex rk-sav} has $p$-th order accuracy ($p \geq 2$), 
		then $\gamma_n$ given in \eqref{imex rrk-sav}, satisfies
		\begin{align}\label{T:gamma_n}
			\gamma_n = 1+\mathcal{O}(\tau^{p-1}),
		\end{align}
		for small enough time step size $\tau$.
	\end{theorem}
	\begin{proof}
		In the specific case of $\frac{\epsilon^2}{2} \mynormB{\nabla \braB{\sum\limits_{i=1}^s  b_i L_{ni}+ \sum\limits_{i=1}^s \bar b_i N_{ni}}}^2 + \braB{\sum\limits_{i=1}^s \bar b_i\tilde{N}_{ni}}^2=0$, according to the definition of $\gamma_n$ in \eqref{Def_gamma_n}, \eqref{T:gamma_n} holds for small enough time step $\tau$.
		
		In other cases, from \eqref{Sn(gamma)}, since $G_n(\gamma_n)$ is quadratic with respect to $\gamma$ and satisfies
		\[
		\begin{aligned}
			G_n(0) &= 0, & G_n(1) &= \mathcal{O}(\tau^{p+1}), \\
			G^{'}_n(0) &= -g\tau^2 + \mathcal{O}(\tau^3), & G^{'}_n(1) &= g\tau^2 + \mathcal{O}(\tau^3),
		\end{aligned}
		\]
		there exists a solution $\gamma_n = 1+\mathcal{O}(\tau^{p-1})$ to the equation $G_n(\gamma_n)$. 
		Thus, $\gamma_n$ given in \eqref{imex rrk-sav} meets \eqref{T:gamma_n}, completing this proof.
	\end{proof}
	
	\subsection{Analysis of the truncation error}
	We now turn our attention to investigate the order of accuracy for the IMEX RRK methods \eqref{imex rrk-sav} 
	based on the estimation of the relaxation coeffcient $\gamma_n$ deduced in the last subsection. 
	Before showing the accuracy, we reformulate the scheme \eqref{imex rrk-sav} into
	\begin{align}\label{imex rrk-sav rewrite}
		\begin{cases}
			U_{ni} = u_{\gamma}^n + \tau \sum\limits_{j=1}^i a_{ij}L_{nj} + \tau \sum\limits_{j=1}^{i-1} \bar a_{ij}N_{nj},\\
			R_{ni} = r_{\gamma}^n + \tau \sum\limits_{j=1}^{i-1} \bar a_{ij}\tilde{N}_j, \quad i=1,\cdots,s,\\
			u^{n+1} = u_{\gamma}^n + \tau \sum\limits_{i=1}^s b_i L_{ni} + \tau \sum\limits_{i=1}^s \bar b_i N_{ni},\\
			r^{n+1} = r_{\gamma}^n + \tau \sum\limits_{i=1}^s \bar b_i\tilde{N}_{ni},\\
			u_{\gamma}^{n+1} = u^{n+1} + (\gamma_n-1)(u^{n+1}-u_{\gamma}^{n}),\\
			r_{\gamma}^{n+1} = r^{n+1} + (\gamma_n-1)(r^{n+1}-r_{\gamma}^{n}).
		\end{cases}
	\end{align}
	Then, following the proof technique in \cite{LiLiZhang2023Implicit,Ranocha2020General}, we can derive the truncation errors of the approximation solution.
	
	\begin{theorem}\label{T:truncation errors}
		Assume that the IMEX RK method \eqref{imex rk-sav} has $p$-th order accuracy ($p \geq 2$), satisfies
		\begin{itemize}
			\item For IDT, i.e., $u_{\gamma}^{n+1}$ and $r_{\gamma}^{n+1}$ approximate the values of $u$ and $r$ at $\hat{t}_{n} + \tau$, respectively, then the IMEX RRK method \eqref{imex rrk-sav} will be of $(p-1)$th order accuracy, namely $u(t_{n+1}) = u^{n+1}_\gamma+\mathcal{O}(\tau^{p})$ and $r(t_{n+1}) = r^{n+1}_\gamma+\mathcal{O}(\tau^{p})$. 
			\item For RT, i.e., $u_{\gamma}^{n+1}$ and $r_{\gamma}^{n+1}$ approximate the values of $u$ and $r$ at $\hat{t}_{n} + \gamma_n \tau$, respectively, then the IMEX RRK method \eqref{imex rrk-sav} will be of $p$th order accuracy, namely $u(t_{n+1}) = u^{n+1}_\gamma+\mathcal{O}(\tau^{p+1})$ and $r(t_{n+1}) = r^{n+1}_\gamma+\mathcal{O}(\tau^{p+1})$.	
		\end{itemize}		
	\end{theorem}
	\begin{proof}
		Using the $p$-order of IMEX RK method \eqref{imex rk-sav} and substituting the exact solution into the third and fourth relations in \eqref{imex rrk-sav rewrite}, we obtain
		\begin{align}\label{ur_p+1}
			\begin{split}
				&u(\hat t_n+\tau)=u(\hat t_n)+\tau\sum\limits_{i=1}^{s}b_i L(u(\hat t_{ni}))+\tau\sum\limits_{i=1}^{s}\bar b_i N(u(\hat t_{ni}),r(\hat t_{ni}))+\mathcal{O}(\tau^{p+1}),\\
				&r(\hat t_n+\tau)=r(\hat t_n)+\tau\sum\limits_{i=1}^{s}\bar b_i \tilde N(u(\hat t_{ni}),r(\hat t_{ni}))+\mathcal{O}(\tau^{p+1}),
			\end{split}		
		\end{align} 
		where $\hat t_{ni}=\hat t_n+c_i\tau,$ ($i=1,\cdots,s$). Then, we estimate the fifth and sixth relations' error in \eqref{imex rrk-sav rewrite}. Plugging the exact solution into the relations and defining 
		\begin{align*}
			&\Psi_1(u,r) := u(\hat t_n)+\tau\sum_{i=1}^{s}b_i L(u(\hat t_{ni}))+\tau\sum_{i=1}^{s}\bar b_i N(u(\hat t_{ni}),r(\hat t_{ni})),\\
			&\Psi_2(u,r) := r(\hat t_n)+\tau\sum\limits_{i=1}^{s}\bar b_i \tilde N(u(\hat t_{ni}),r(\hat t_{ni}))
		\end{align*}
		gives
		\begin{align}\label{psi_equ}
			\begin{split}
				&u(\hat t_{n+1}) = \Psi_1(u,r) + (\gamma_n-1)(\Psi_1(u,r)-u(\hat t_n))+d_{u,n+1},\\
				&r(\hat t_{n+1}) = \Psi_2(u,r) + (\gamma_n-1)(\Psi_2(u,r)-r(\hat t_n))+d_{r,n+1},
			\end{split}		
		\end{align}
		where $d_{u,n+1}$ and $d_{r,n+1}$ denote the errors between ($u(\hat t_{n+1})$, $u^{n+1}_\gamma$) and ($r(\hat t_{n+1})$, $r^{n+1}_\gamma$), respectively, and are calculated by IMEX RRK methods \eqref{imex rrk-sav}. Hence, using \eqref{T:gamma_n}, \eqref{ur_p+1} and \eqref{psi_equ}, we get
		\begin{align*}
			d_{u,n+1} &= u(\hat t_{n+1}) - \Psi_1(u,r) - (\gamma_n-1) \brab{\Psi(u,r) - u(\hat t_n)}\\
			&= u(\hat t_{n+1})-u(\hat t_n+\tau)-\tau u^{'}(\hat t_n+\tau) \mathcal{O}(\tau^{p-1})+\mathcal{O}(\tau^{p+1}),
		\end{align*}
		and similarly, the scalar auxiliary variable $r$ has the same result, namely
		\begin{align*}
			d_{r,n+1} = r(\hat t_{n+1})-r(\hat t_n+\tau)-\tau r^{'}(\hat t_n+\tau) \mathcal{O}(\tau^{p-1})+\mathcal{O}(\tau^{p+1}).
		\end{align*}
		
		For IDT, in other words $\hat t_{n+1}= \hat t_{n}+\tau$. 
		Applying the estimate of $\gamma_n$, we have 
		\begin{align*}
			d_{u,n+1}=-\tau u^{'}(\hat t_n+\tau) \mathcal{O}(\tau^{p-1})+\mathcal{O}(\tau^{p+1})=\mathcal{O}(\tau^{p})
		\end{align*}
		and $d_{r,n+1}=\mathcal{O}(\tau^{p}),$
		that is the IMEX RRK method \eqref{imex rrk-sav} achieves $(p-1)$th order accuracy. 
		
		For RT, in other words $\hat t_{n+1}= \hat t_{n}+\gamma_n \tau$. 
		Applying the estimate of $\gamma_n$ and Taylor expansion, we get
		\begin{align*}
			d_{u,n+1}&=u\big(\hat t_{n}+\tau+(\gamma_n-1)\tau\big)-u(\hat t_n+\tau)- \tau u^{'}(\hat t_n+\tau) (\gamma_n-1)+\mathcal{O}(\tau^{p+1}) = \mathcal{O}(\tau^{p+1})
		\end{align*}
		and $d_{r,n+1}=\mathcal{O}(\tau^{p+1})$,
		that is the IMEX RRK method \eqref{imex rrk-sav} under RT achieves $p$th order accuracy. 
		
	\end{proof}

	\section{Numerical tests} 
	In this section, we present several numerical examples to validate the theoretical results for the IMEX RRK methods \eqref{imex rrk-sav} presented in the above. Due to the high accuracy of the Fourier pseudo-spectral method, we employ this method with a uniform grid scale of $128 \times 128$ for the spatial discretization to ensure that 
	the spatial error is negligible compared to the temporal error. Following the idea of the proof of Theorem 4.1 in \cite{Yang2022Arbitrarily}, the modified energy of the spatial semi-discretization also satisfies the energy dissipation law.
	
	\begin{example}\label{example1}
		{\rm Consider the Allen-Cahn (AC) equation}
		\[
		u_t - \epsilon^2 \Delta u + u^3 - u = 0, \quad (x, y) \in (0, 2\pi)^2, \quad t \in (0, 1].
		\]
	\end{example}
	Firstly, we denote IMEX RRK$(s,p)$ be an $s$-stage, $p$-th order implicit-explicit relaxation RK method (see appendix for coefficients) and set $\epsilon= 0.5, C_0=0$ with the initial condition  $u_0 = 0.5 \sin x \sin y$.
	We investigate the relaxation coeffcient $\gamma_n$ at every step by computing $\mynorm{\gamma _n-1}_{\ell^{\infty}}$ (i.e. the maximum distances between $\gamma$ and $1$) and $\mynorm{G_n(1)}_{\ell^{\infty}}$ (i.e. the maximum values of $G_n(1)$) for various time step sizes. 
	In Figures \ref{F:Ex1_Gamma} and \ref{F:Ex1_Sn}, IMEX RRK$(s,p)$ have $(p-1)$-order for $\mynorm{\gamma _n-1}_{\ell^{\infty}}$ and $(p+1)$-order for $\mynorm{G_n(1)}_{\ell^{\infty}}$, 
	which supports our analytical results (Lemma \ref{L:Sn(1)} and Theorem \ref{T:IMEXRRK gamma}). 
	We would like to point out that the different interpretations (IDT and RT) yield the same results for both $\mynorm{\gamma _n-1}_{\ell^{\infty}}$ and $\mynorm{G_n(1)}_{\ell^{\infty}}$.
	Besides, we illustrate the modified energy dissipation property of the IMEX RRK$(3,2)$ method. The tiem step size is $\tau=10^{-3}$ and the final time is $T_0=5$. Figure \ref{F:energy_ac} shows that the IMEX RRK method \eqref{imex rrk-sav} satisfies the energy dissipation law.
	
	\begin{figure}
		\includegraphics[width=0.88\textwidth]{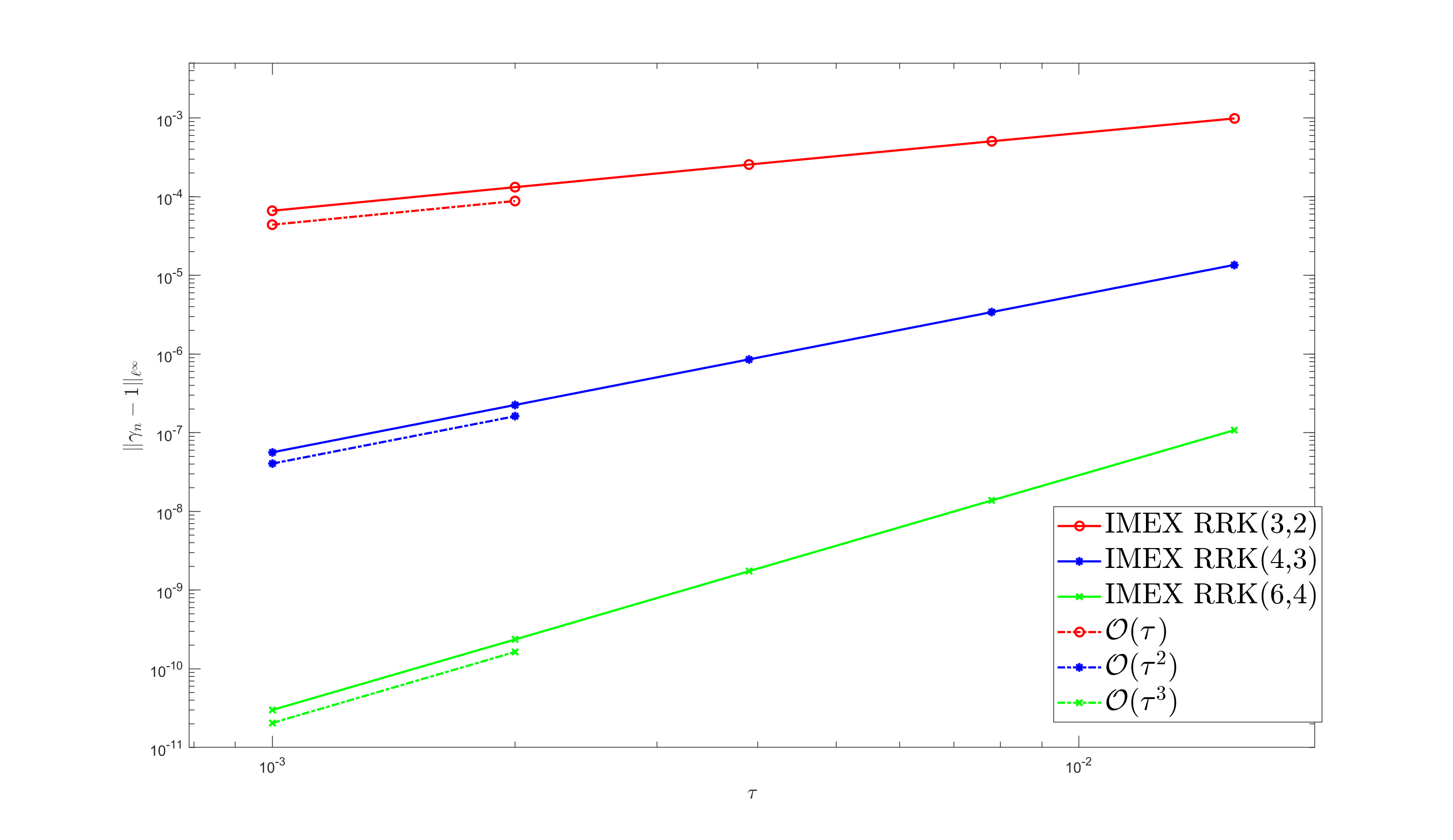}
		\caption{\label{F:Ex1_Gamma} $\mynorm{\gamma _n-1}_{\ell^{\infty}}$ for some relaxation (RT) methods (AC equation)} 
	\end{figure} 
	
	\begin{figure}
		\includegraphics[width=0.88\textwidth]{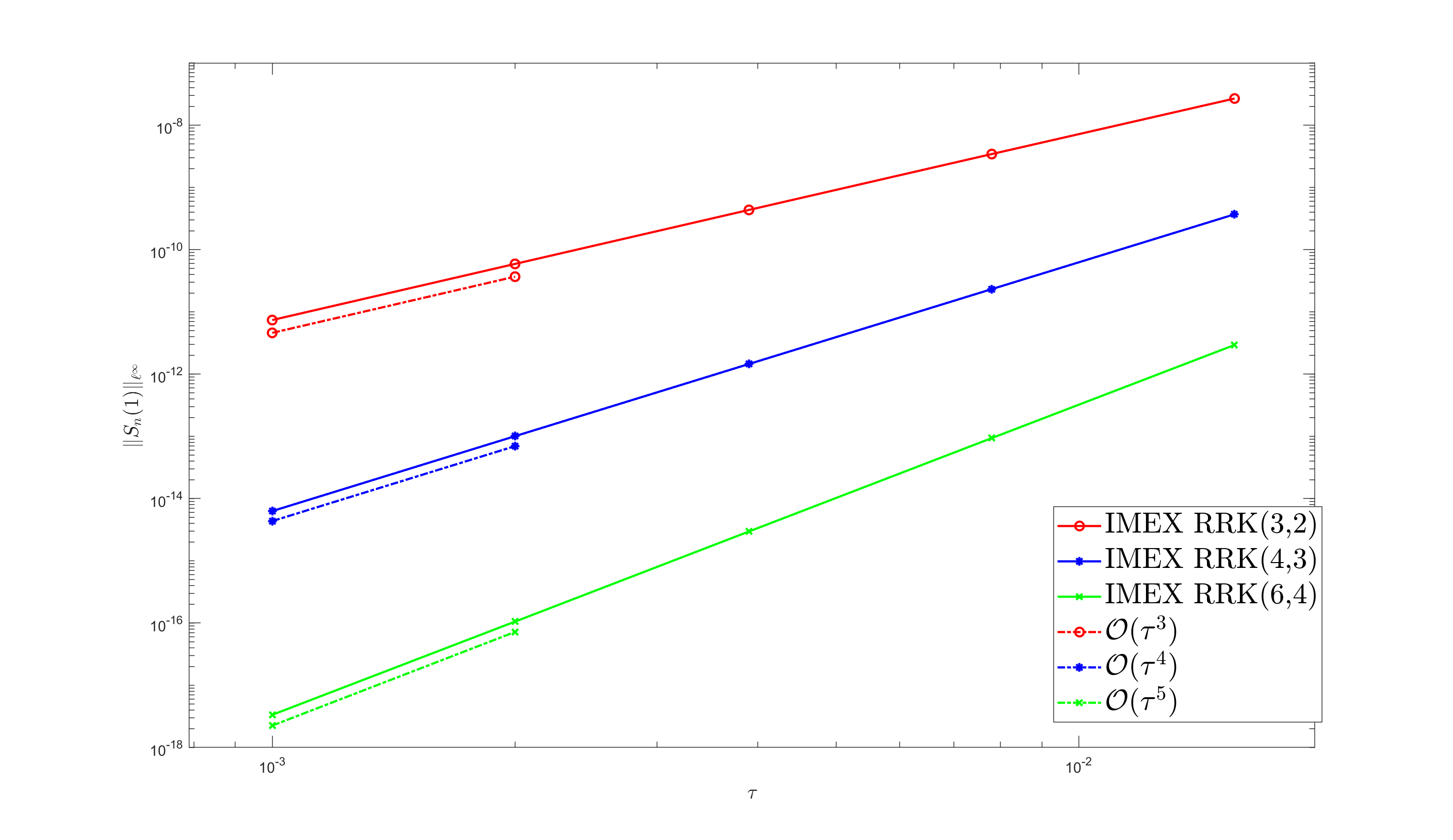}
		\caption{\label{F:Ex1_Sn} $\mynorm{G_n(1)}_{\ell^{\infty}}$ for some relaxation (RT) methods (AC equation)} 
	\end{figure} 

	\begin{figure}
		\includegraphics[width=0.88\textwidth]{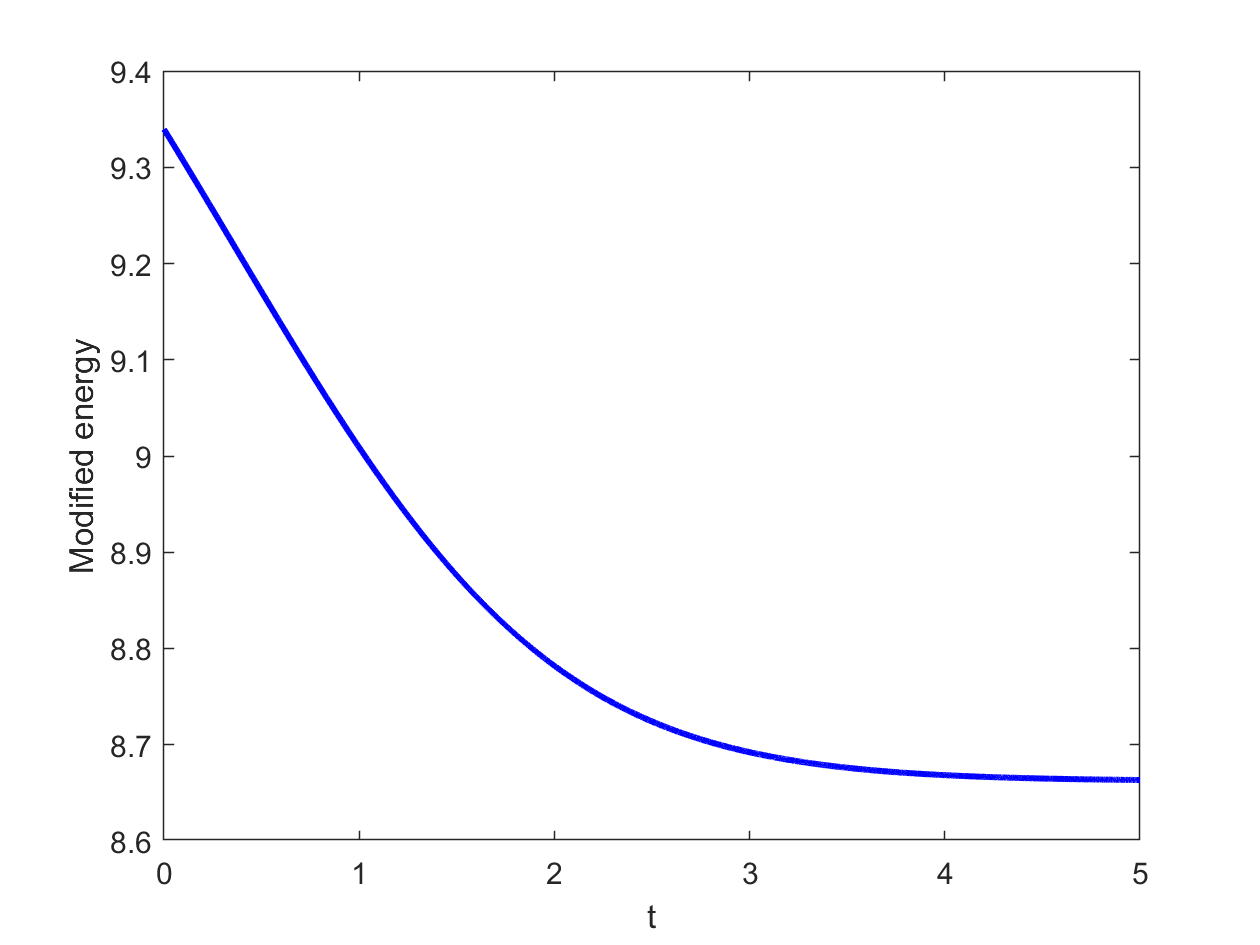}
		\caption{\label{F:energy_ac} Discrete energy evolution (AC equation)} 
	\end{figure} 
	
	Secondly, we show the accuracy of the approximation solution from the IMEX RRK method \eqref{imex rrk-sav} using a reference solution with a very small time step. We measure the errors ($\ell^{\infty}$-norm) and convergence orders at $T_0=1$. The results, shown in Table \ref{Table:EX1 ac-error}, 
	reveal that the IDT methods get $(p-1)$th order accuracy, which is one order lower than standard methods, whereas RT methods get $p$th order accuracy, matching the standard methods. 
	For example, IMEX RRK$(3,2)$ has first-order accuracy for IDT and second-order accuracy for RT, which confirms the Theorem \ref{T:truncation errors}.
	
	\begin{small}
		\begin{table}
			\caption{\label{Table:EX1 ac-error} $\ell^{\infty}$-norm errors and convergence orders (AC equation)}
			\begin{tabular}{|c|c|c|c|c|c|cc|cc|}
				\hline
				\multicolumn{2}{|c}{\multirow{2}*{$p$}} & \multicolumn{2}{|c}{\multirow{2}*{Method}} & \multicolumn{2}{|c|}{\multirow{2}*{$\tau$}} &  \multicolumn{2}{c|}{IDT} & \multicolumn{2}{c|}{RT}\\
				\multicolumn{2}{|c}{~} & \multicolumn{2}{|c}{~} & \multicolumn{2}{|c|}{~} & Error & Order & Error & Order\\
				\hline
				\multicolumn{2}{|c}{\multirow{4}*{2}} & \multicolumn{2}{|c}{\multirow{4}*{IMEX RRK(3,2)}} & \multicolumn{2}{|c|}{$1/100$} &  5.5774e-05 & - & 4.3957e-07  & - \\
				\multicolumn{2}{|c}{\multirow{4}*{~}} & \multicolumn{2}{|c}{\multirow{4}*{~}} & \multicolumn{2}{|c|}{$1/200$} &  2.8253e-05  & 0.9812 & 1.1321e-07 & 1.9571 \\
				\multicolumn{2}{|c}{\multirow{4}*{~}} & \multicolumn{2}{|c}{\multirow{4}*{~}} & \multicolumn{2}{|c|}{$1/400$} &  1.4195e-05  & 0.9930 & 2.8732e-08 & 1.9783 \\
				\multicolumn{2}{|c}{\multirow{4}*{~}} & \multicolumn{2}{|c}{\multirow{4}*{~}} & \multicolumn{2}{|c|}{$1/800$} &  7.0929e-06  & 1.0009 & 7.2373e-09 & 1.9891 \\
				\hline
				\multicolumn{2}{|c}{\multirow{4}*{3}} & \multicolumn{2}{|c}{\multirow{4}*{IMEX RRK(4,3)}} & \multicolumn{2}{|c|}{$1/100$} &  5.1766e-07  & - & 5.1005e-09 & - \\
				\multicolumn{2}{|c}{\multirow{4}*{~}} & \multicolumn{2}{|c}{\multirow{4}*{~}} & \multicolumn{2}{|c|}{$1/200$} &  1.2935e-07  & 2.0008 & 6.4715e-10  & 2.9784 \\
				\multicolumn{2}{|c}{\multirow{4}*{~}} & \multicolumn{2}{|c}{\multirow{4}*{~}} & \multicolumn{2}{|c|}{$1/400$} &  3.2327e-08  & 2.0004 & 8.1502e-11  & 2.9892 \\
				\multicolumn{2}{|c}{\multirow{4}*{~}} & \multicolumn{2}{|c}{\multirow{4}*{~}} & \multicolumn{2}{|c|}{$1/800$} &  8.0811e-09  & 2.0001 & 1.0224e-11  & 2.9949 \\
				\hline
				\multicolumn{2}{|c}{\multirow{4}*{4}} & \multicolumn{2}{|c}{\multirow{4}*{IMEX RRK(6,4)}} & \multicolumn{2}{|c|}{$1/16$} &  6.0853e-07  & - & 1.8273e-08  & - \\
				\multicolumn{2}{|c}{\multirow{4}*{~}} & \multicolumn{2}{|c}{\multirow{4}*{~}} & \multicolumn{2}{|c|}{$1/32$} &  8.0790e-08  & 2.9131 & 1.2800e-09  & 3.8356 \\
				\multicolumn{2}{|c}{\multirow{4}*{~}} & \multicolumn{2}{|c}{\multirow{4}*{~}} & \multicolumn{2}{|c|}{$1/64$} &  1.0422e-08  & 2.9546 & 8.5163e-11  & 3.9097 \\
				\multicolumn{2}{|c}{\multirow{4}*{~}} & \multicolumn{2}{|c}{\multirow{4}*{~}} & \multicolumn{2}{|c|}{$1/128$} &  1.3239e-09  & 2.9768 & 5.5013e-12  & 3.9524 \\
				\hline
			\end{tabular}
		\end{table}
	\end{small}
	
	Lastly, we show the efficiency of the proposed scheme from the perspective of phase separation by using the IMEX RRK(3,2) method.
	Here, we set $ u_0 = 0.001 \cdot \text{Rand}(x, y) $, $\tau = 10^{-3}$, and $\epsilon=0.005$. Figure \ref{F:Ex1_evolution} shows the coarsening evolution, displaying interconnected patterns where yellow and blue regions represent $u = 1$ and $u = -1$, respectively. In the beginning, the image exhibits a certain degree of randomness. With the development of time, the phases will gradually merge, forming a shape at approximately $t = 40$, and then become stable eventually.

	\begin{figure}
		\includegraphics[width=0.32\textwidth]{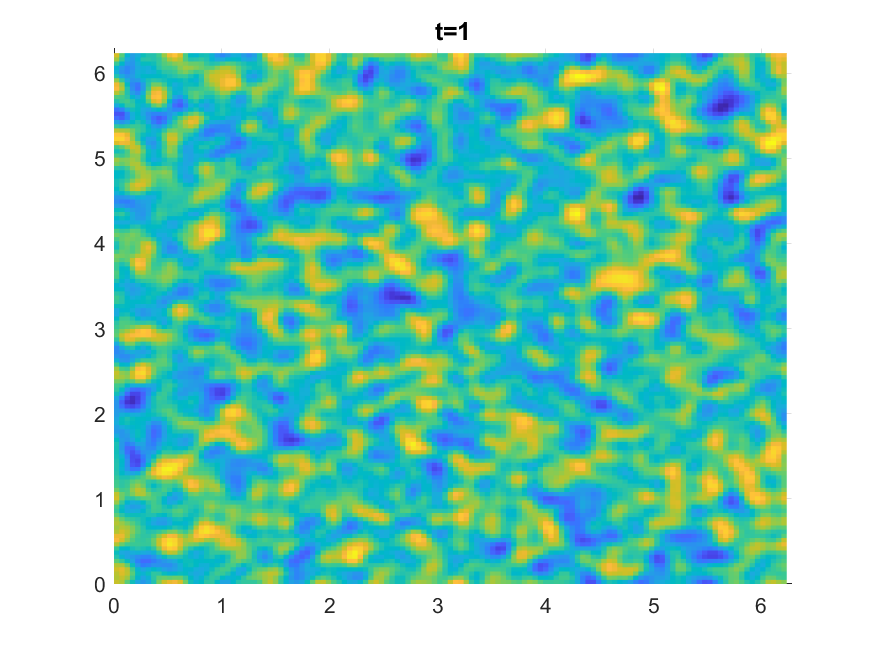}
		\includegraphics[width=0.32\textwidth]{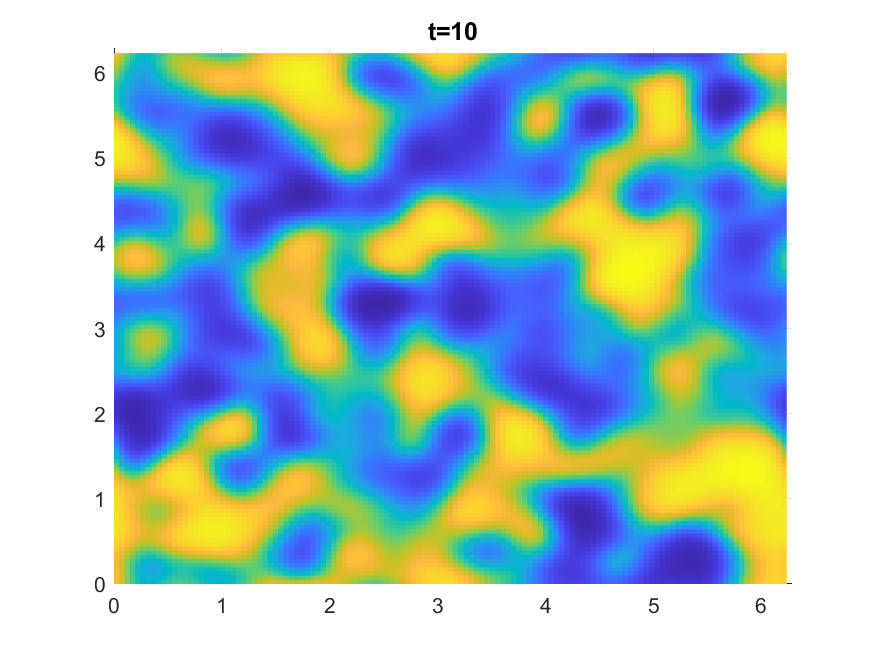}
		\includegraphics[width=0.32\textwidth]{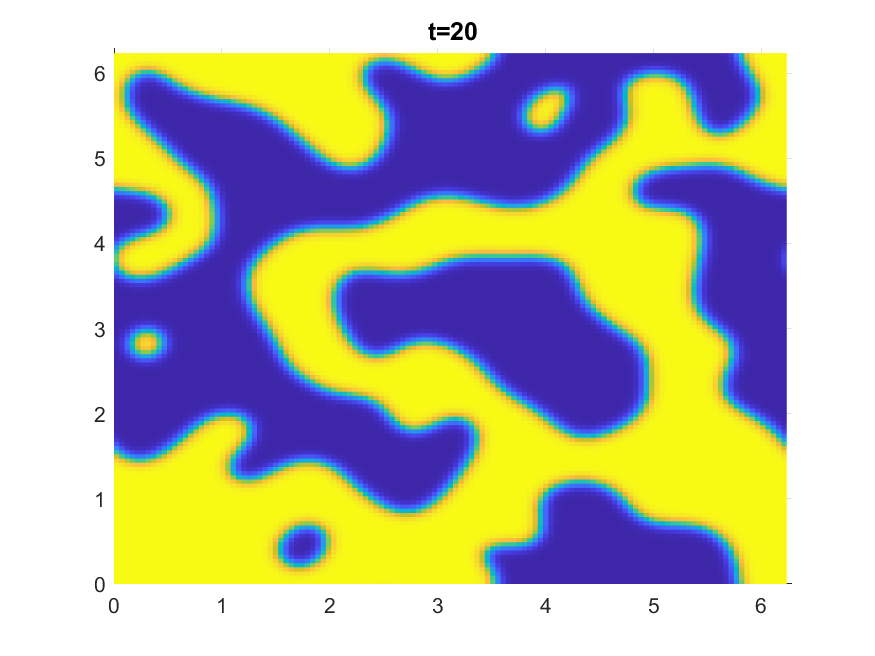}
		\includegraphics[width=0.32\textwidth]{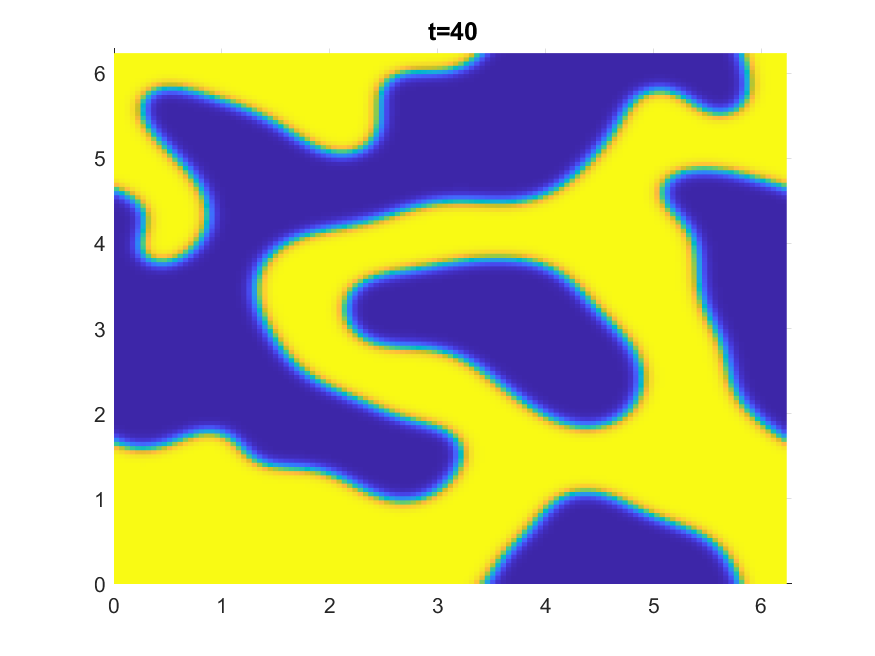}
		\includegraphics[width=0.32\textwidth]{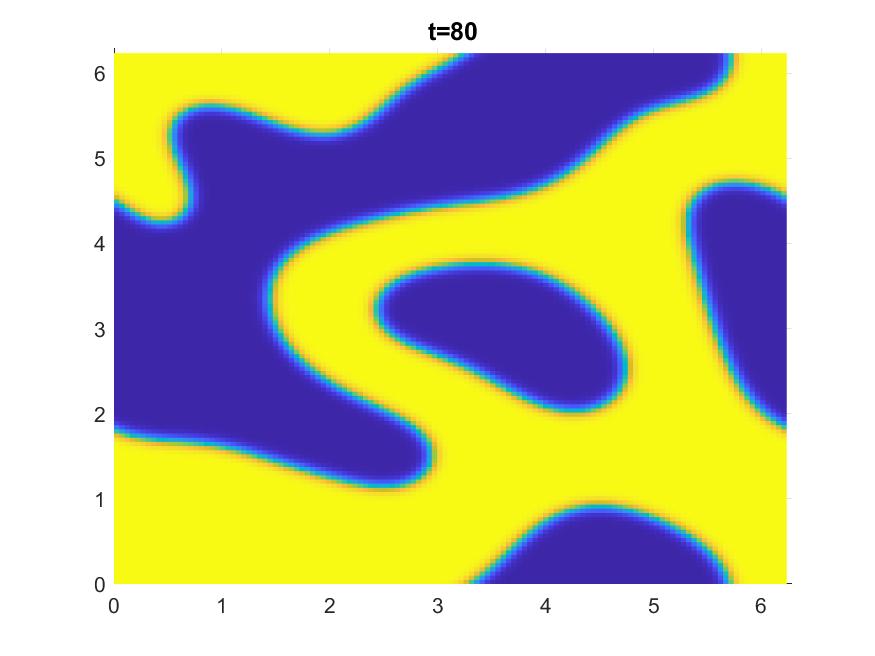}
		\includegraphics[width=0.32\textwidth]{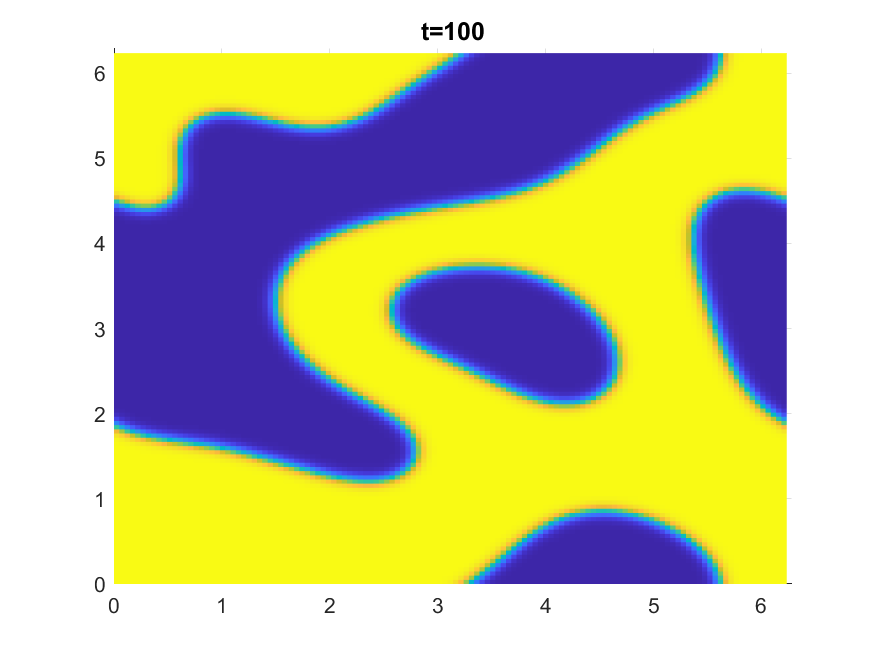}
		\caption{\label{F:Ex1_evolution}Solution snapshots for phase separation in the AC equation}
	\end{figure}

	\begin{example}\label{example2}
		{\rm Consider the Cahn-Hilliard (CH) equation}
		\[
		u_t + \Delta(\epsilon^2 \Delta u - u^3 + u) = 0, \quad (x, y) \in (0, 2\pi)^2, \quad t \in (0, 1].
		\]
	\end{example}
	
	Setting $\epsilon= 1$ and the initial condition $u_0 = 0.5\sin x \sin y$, we calculate $\mynorm{\gamma _n-1}_{\ell^{\infty}}$ and $\mynorm{G_n(1)}_{\ell^{\infty}}$ with the IMEX RRK method \eqref{imex rrk-sav} for different time step sizes. The results shown in Figures \ref{F:Ex2_Gamma} and \ref{F:Ex2_Sn} also validate the Lemma \ref{L:Sn(1)} and Theorem \ref{T:IMEXRRK gamma}, respectively. Since IDT and RT lead to the same conclusions for $\mynorm{\gamma _n-1}_{\ell^{\infty}}$ and $\mynorm{G_n(1)}_{\ell^{\infty}}$, 
	we omit the results regarding to the IDT method for the purpose of conciseness.
	Besides, we show the energy dissipation property of the IMEX RRK$(3,2)$ method. The tiem step size is $\tau=10^{-2}$ and the final time is $T_0=5$. Figure \ref{F:energy_ch} shows that the modified energy decays, which is consistent with our theoretical result.
	
	\begin{figure}
		\includegraphics[width=0.88\textwidth]{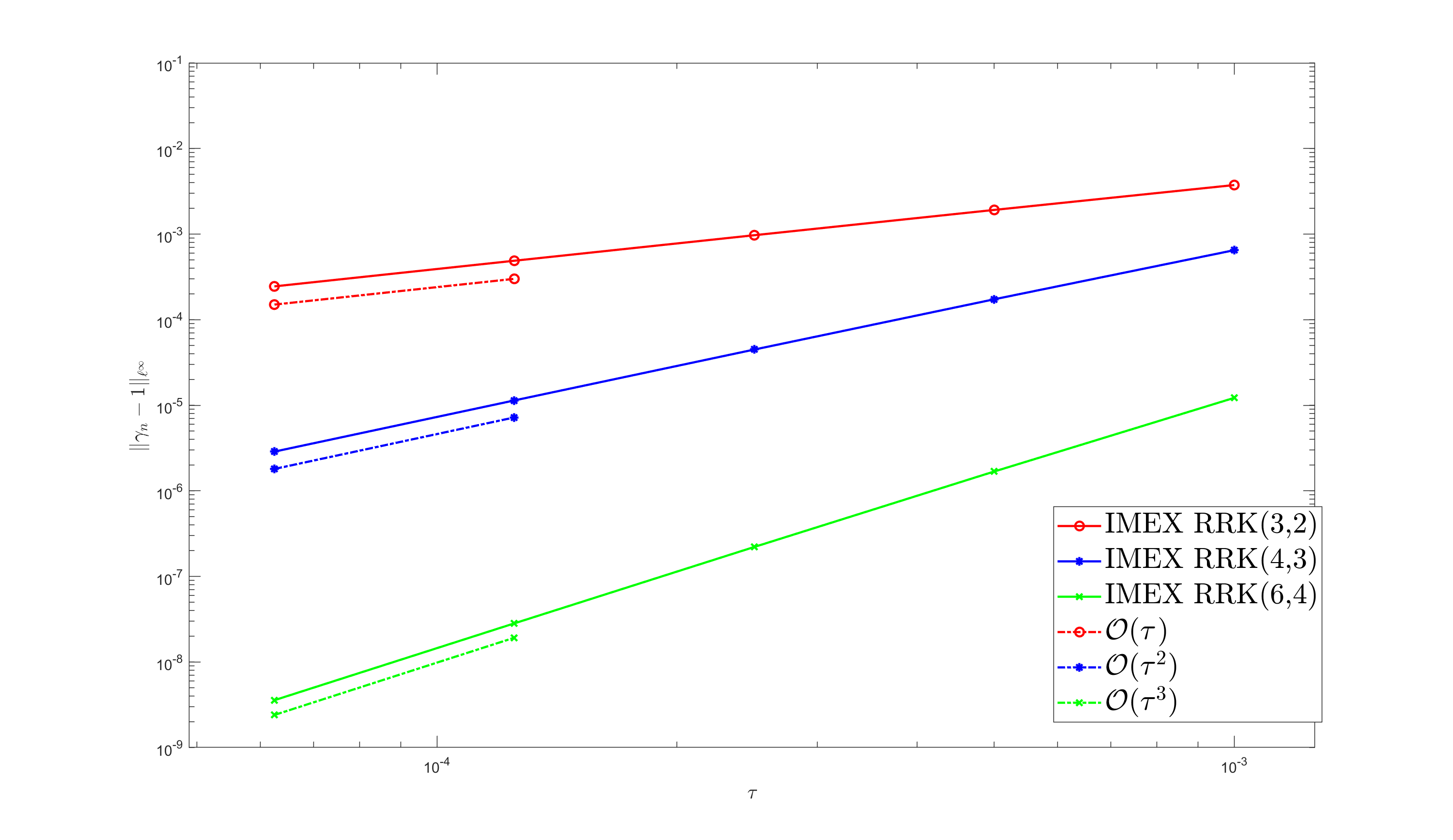}
		\caption{\label{F:Ex2_Gamma} $\mynorm{\gamma _n-1}_{\ell^{\infty}}$ for some relaxation (RT) methods (CH equation)}
	\end{figure}
	
	\begin{figure}
		\includegraphics[width=0.88\textwidth]{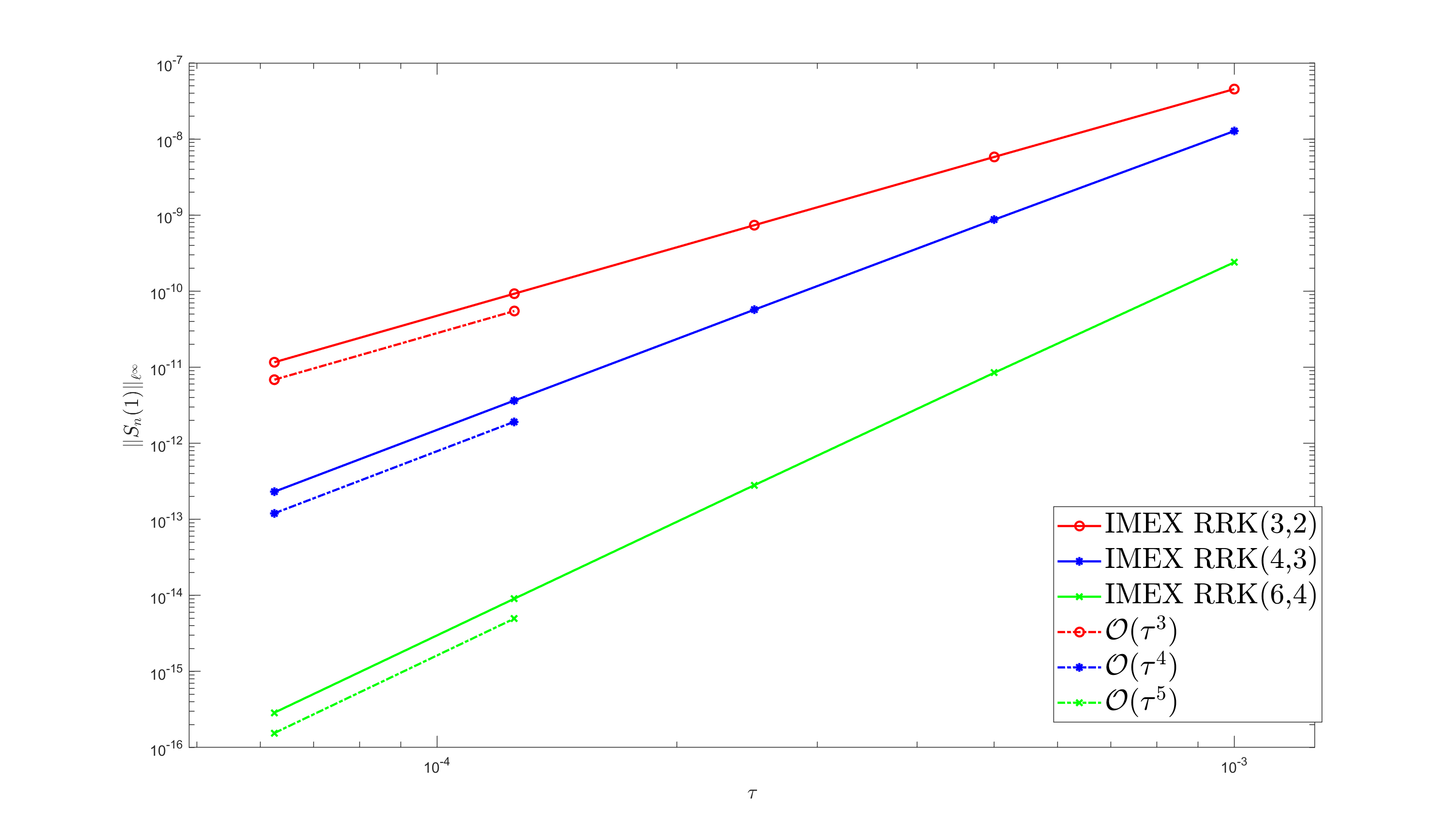}
		\caption{\label{F:Ex2_Sn} $\mynorm{G_n(1)}_{\ell^{\infty}}$ for some relaxation (RT) methods (CH equation)}
	\end{figure}
	
	\begin{figure}
		\includegraphics[width=0.88\textwidth]{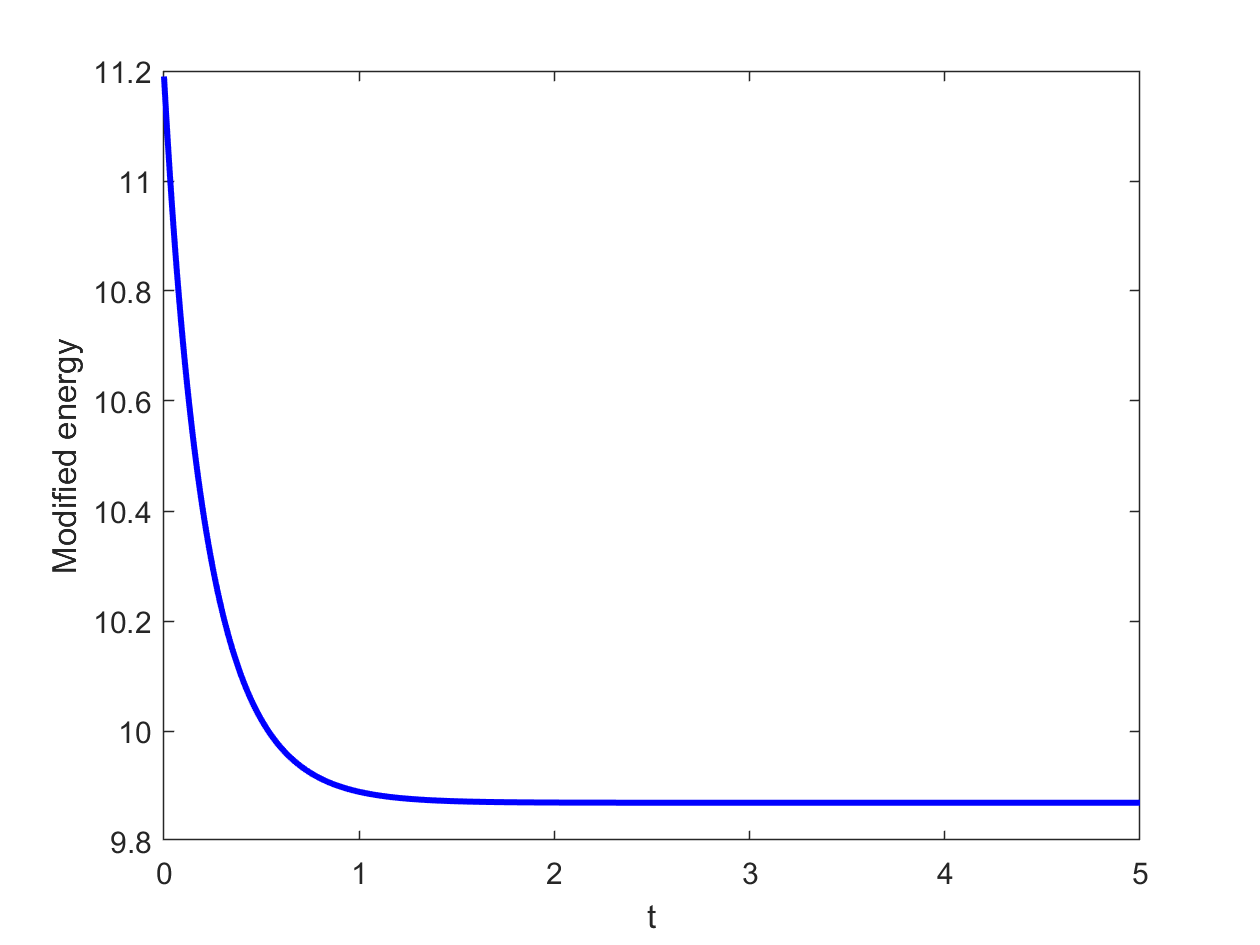}
		\caption{\label{F:energy_ch} Discrete energy evolution (CH equation)} 
	\end{figure} 

	Then, we also study the convergence orders of the approximation solution and compute a reference solution with a very small time step. 
	The $\ell^{\infty}$-norm at $T_0=1$ in Table \ref{Table:EX2 ch-error} shows the IDT methods 
	have $(p-1)$th order accuracy and the RT methods have $p$th order accuracy, 
	supporting the analytical conclusions in theorem \ref{T:truncation errors}.
	
	\begin{small}
		\begin{table}
			\caption{\label{Table:EX2 ch-error} $\ell^{\infty}$-norm errors and convergence orders (CH equation)}
			\begin{tabular}{|c|c|c|c|c|c|cc|cc|}
				\hline
				\multicolumn{2}{|c}{\multirow{2}*{$p$}} & \multicolumn{2}{|c}{\multirow{2}*{Method}} & \multicolumn{2}{|c|}{\multirow{2}*{$\tau$}} &  \multicolumn{2}{c|}{IDT} & \multicolumn{2}{c|}{RT}\\
				\multicolumn{2}{|c}{~} & \multicolumn{2}{|c}{~} & \multicolumn{2}{|c|}{~} & Error & Order & Error & Order\\
				\hline
				\multicolumn{2}{|c}{\multirow{4}*{2}} & \multicolumn{2}{|c}{\multirow{4}*{IMEX RRK(3,2)}} & \multicolumn{2}{|c|}{$1e^{-3}$} &  8.0228e-05 & - & 5.7195e-08  & - \\
				\multicolumn{2}{|c}{\multirow{4}*{~}} & \multicolumn{2}{|c}{\multirow{4}*{~}} & \multicolumn{2}{|c|}{$1e^{-3}/2$} &  4.0513e-05  & 0.9857 & 1.4954e-08 & 1.9354 \\
				\multicolumn{2}{|c}{\multirow{4}*{~}} & \multicolumn{2}{|c}{\multirow{4}*{~}} & \multicolumn{2}{|c|}{$1e^{-3}/4$} &  2.0048e-05  & 1.0149 & 3.8200e-09 & 1.9689 \\
				\multicolumn{2}{|c}{\multirow{4}*{~}} & \multicolumn{2}{|c}{\multirow{4}*{~}} & \multicolumn{2}{|c|}{$1e^{-3}/8$} &  9.6568e-06  & 1.0539 & 9.6136e-10 & 1.9904 \\
				\hline
				\multicolumn{2}{|c}{\multirow{4}*{3}} & \multicolumn{2}{|c}{\multirow{4}*{IMEX RRK(4,3)}} & \multicolumn{2}{|c|}{$1e^{-3}$} &  4.8085e-07  & - & 1.0307e-09 & - \\
				\multicolumn{2}{|c}{\multirow{4}*{~}} & \multicolumn{2}{|c}{\multirow{4}*{~}} & \multicolumn{2}{|c|}{$1e^{-3}/2$} &  1.0727e-07  & 2.1643 & 1.2426e-10  & 3.0523 \\
				\multicolumn{2}{|c}{\multirow{4}*{~}} & \multicolumn{2}{|c}{\multirow{4}*{~}} & \multicolumn{2}{|c|}{$1e^{-3}/4$} &  2.4985e-08  & 2.1022 & 1.5199e-11  & 3.0313 \\
				\multicolumn{2}{|c}{\multirow{4}*{~}} & \multicolumn{2}{|c}{\multirow{4}*{~}} & \multicolumn{2}{|c|}{$1e^{-3}/8$} &  5.9757e-09  & 2.0639 & 1.8767e-12  & 3.0177 \\
				\hline
				\multicolumn{2}{|c}{\multirow{4}*{4}} & \multicolumn{2}{|c}{\multirow{4}*{IMEX RRK(6,4)}} & \multicolumn{2}{|c|}{$8e^{-3}$} &  2.1762e-06  & - & 2.2700e-08  & - \\
				\multicolumn{2}{|c}{\multirow{4}*{~}} & \multicolumn{2}{|c}{\multirow{4}*{~}} & \multicolumn{2}{|c|}{$8e^{-3}/2$} &  2.9870e-07  & 2.8650 & 1.6035e-09  & 3.8234 \\
				\multicolumn{2}{|c}{\multirow{4}*{~}} & \multicolumn{2}{|c}{\multirow{4}*{~}} & \multicolumn{2}{|c|}{$8e^{-3}/4$} &  3.6762e-08  & 3.0224 & 1.0141e-10  & 3.9830 \\
				\multicolumn{2}{|c}{\multirow{4}*{~}} & \multicolumn{2}{|c}{\multirow{4}*{~}} & \multicolumn{2}{|c|}{$8e^{-3}/8$} &  4.4734e-09  & 3.0388 & 6.2989e-12  & 4.0089 \\
				\hline
			\end{tabular}
		\end{table}
	\end{small}
	
	Finally, we also simulate the phase separation processes using the IMEX RRK(3,2) method. 
	Let $u_0 = 0.4 \cdot \text{Rand}(x, y) + 0.25$, $\tau = 10^{-5}$ and $\epsilon=0.1$, with the results displayed in Figure \ref{F:Ex2_evolution}. As seen in the figures, we note that the distinct phases also gradually merge and reach a stable state.
	
	\begin{figure}
		\includegraphics[width=0.30\textwidth]{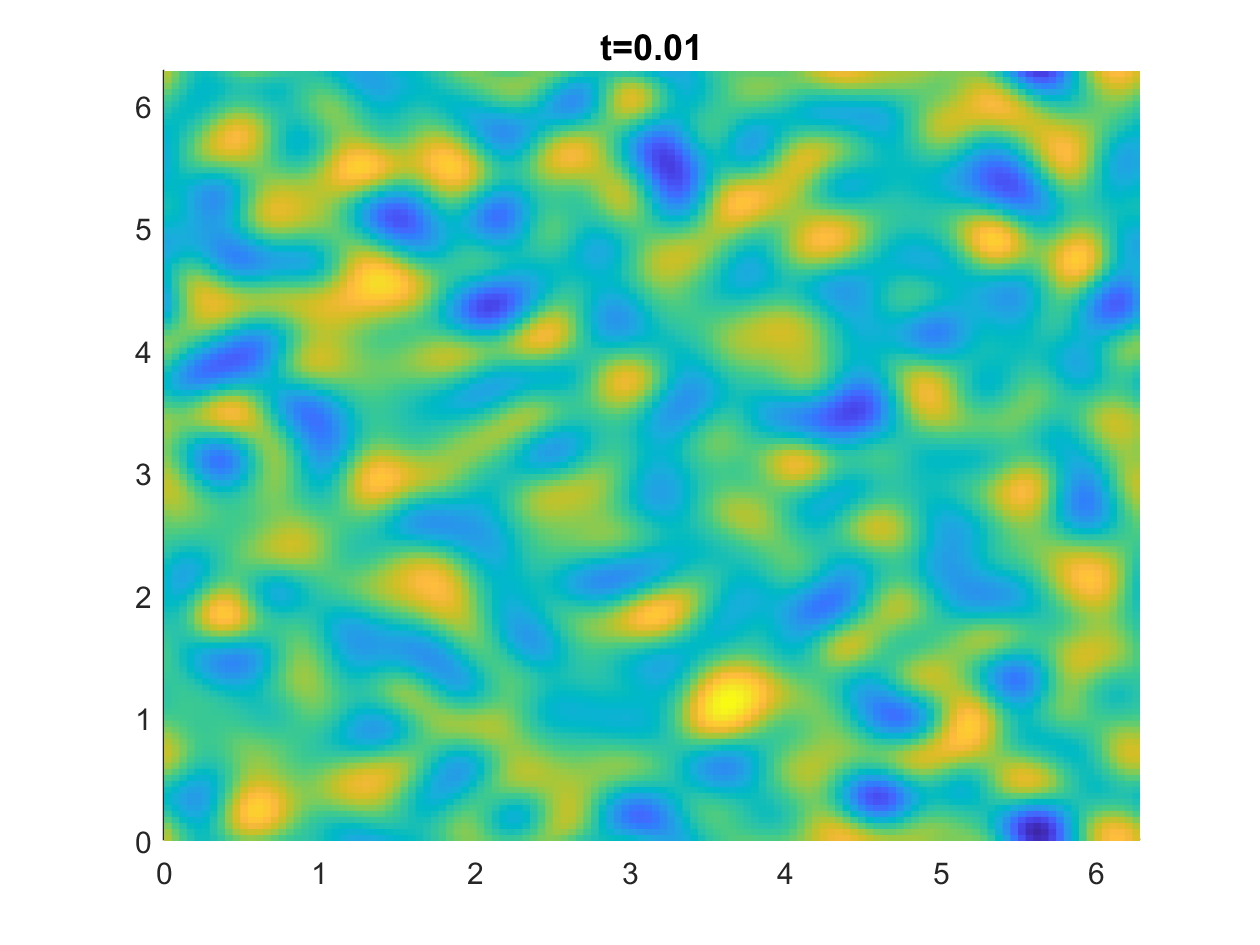}
		\includegraphics[width=0.30\textwidth]{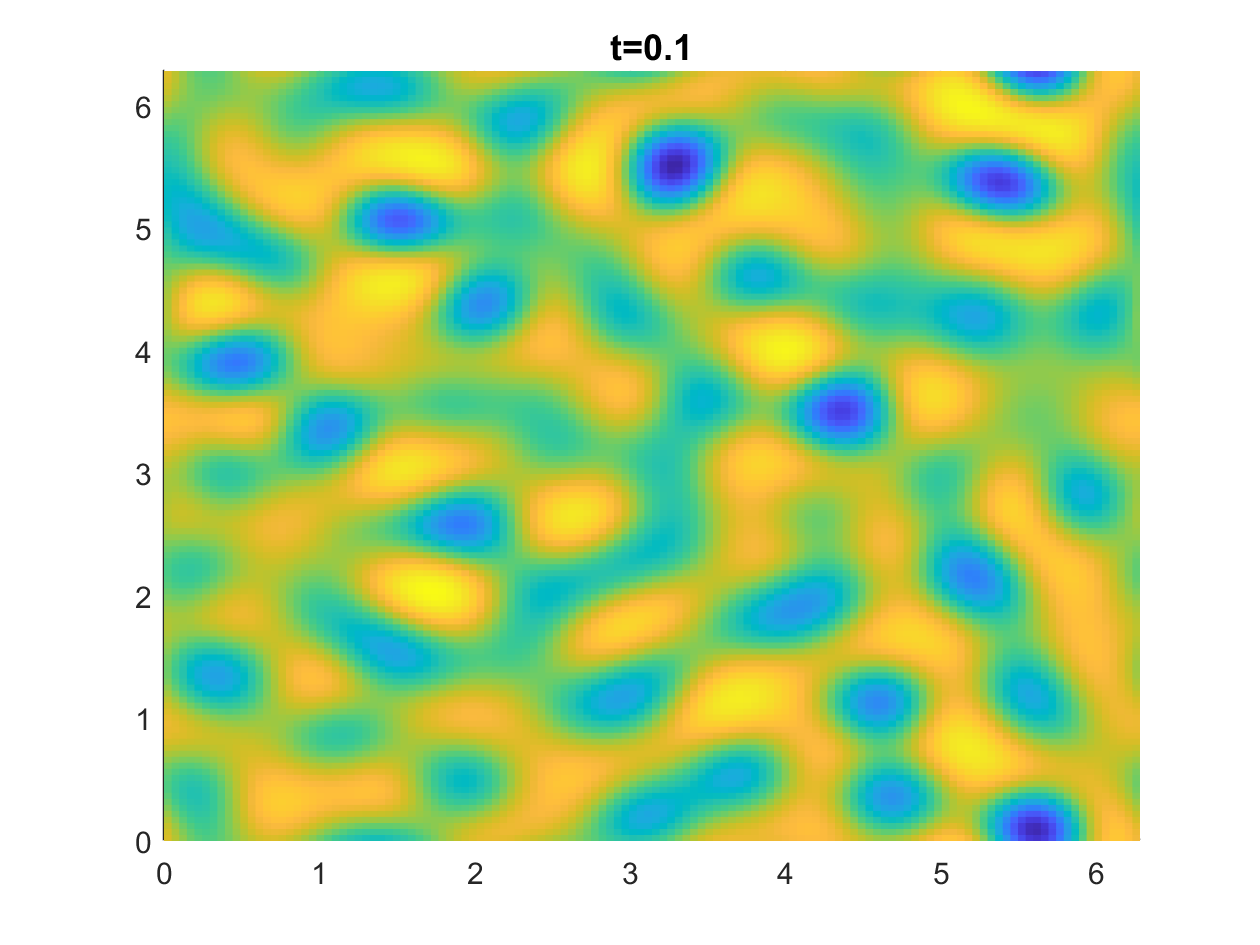}
		\includegraphics[width=0.30\textwidth]{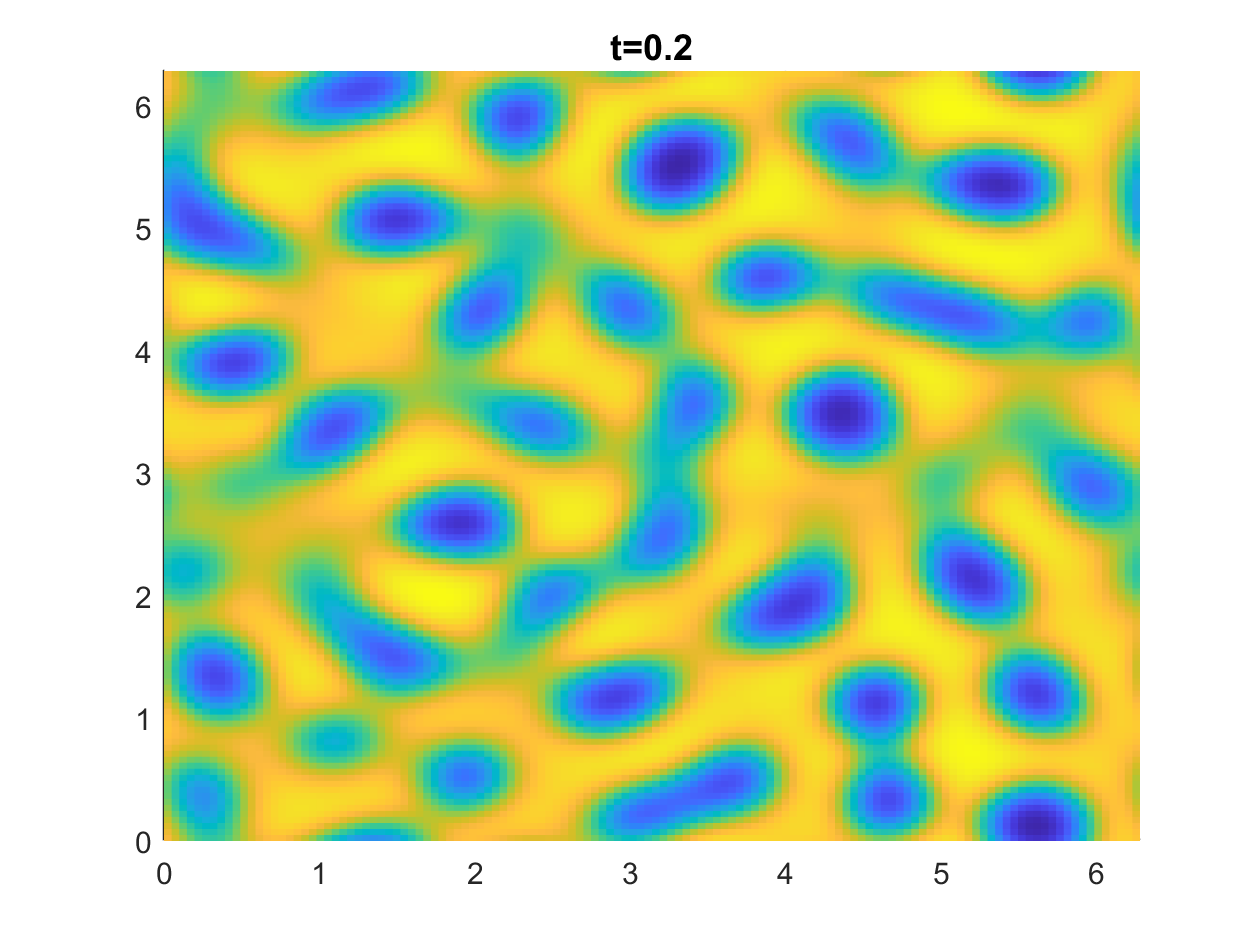}
		\includegraphics[width=0.30\textwidth]{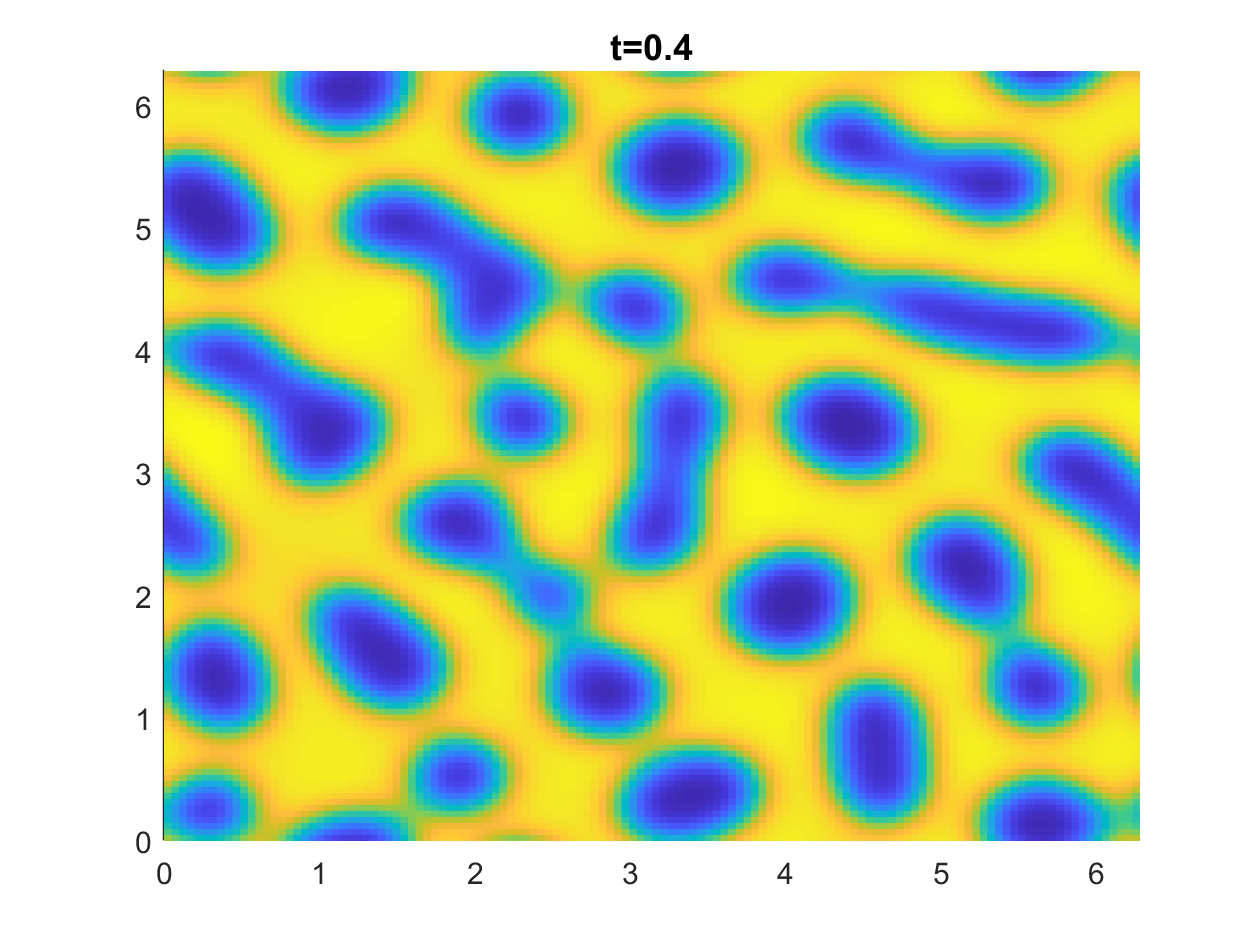}
		\includegraphics[width=0.30\textwidth]{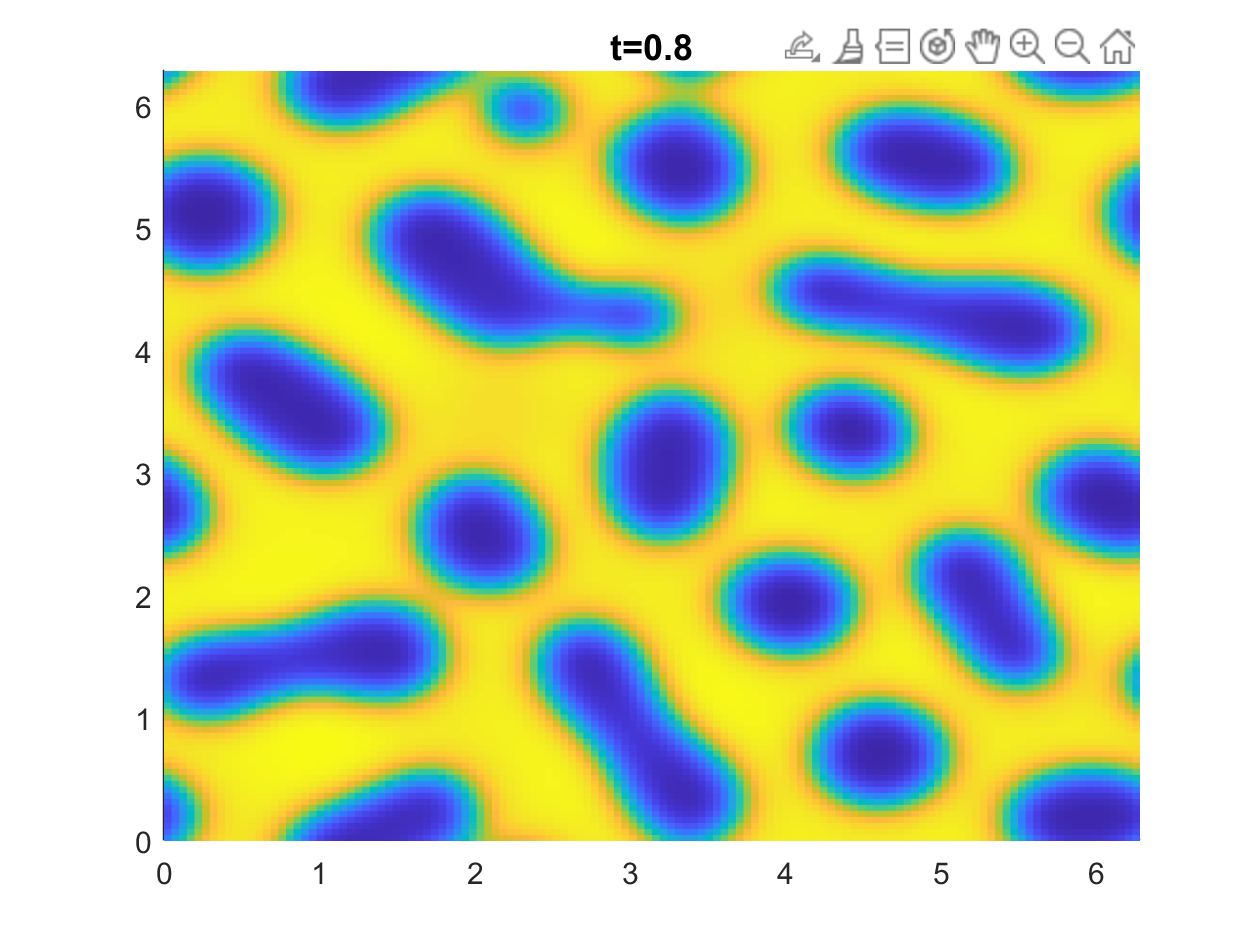}
		\includegraphics[width=0.30\textwidth]{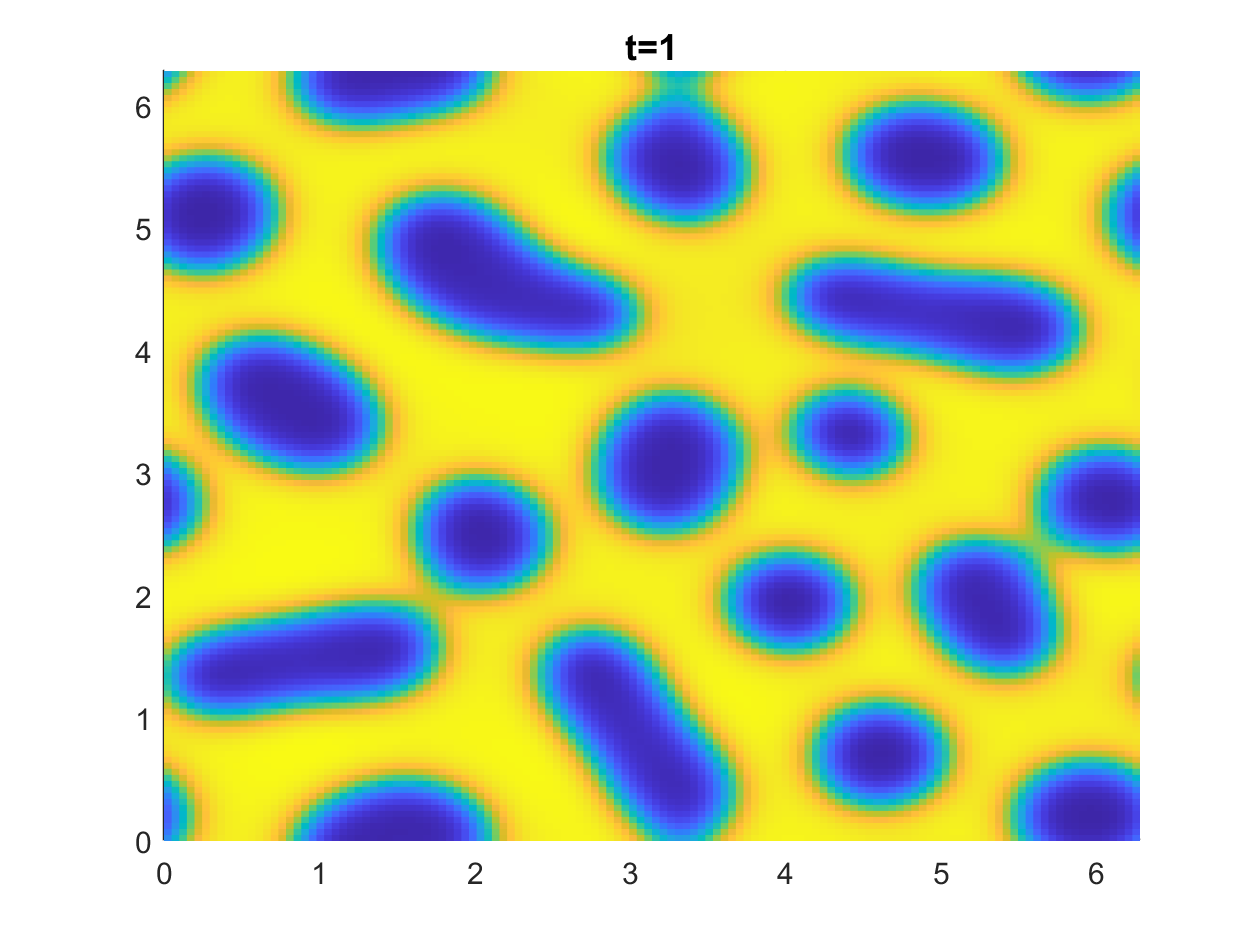}
		\caption{\label{F:Ex2_evolution} Solution snapshots for phase separation in the CH equation}
	\end{figure}

	\begin{example}\label{example3}
		{\rm \textbf{(Extension)} In this example, we consider the multi-component gradient flows, whose energy functional (cf. \cite{Shen2018scalar}) is defined as follows}
		\begin{align}
			{E}(u_1,u_2,\cdots ,u_k) = \sum_{\ell=1}^k \int_{\Omega} \braB{\frac{\epsilon^2}{2} |\nabla u_{\ell}|^2 + F(u_{\ell})} \,\zd \x.
		\end{align}
		{\rm where $u_{\ell}\,(\ell=1,2,\cdots  k)$ is the $\ell$-th phase field variable representing the volume fraction of the $\ell$-th component in the fluid mixture.}
	\end{example}
	
	Similarly, we introduce the scalar auxiliary variable $\tilde r(t) = \sqrt{\sum_{\ell=1}^k \int_{\Omega} F(u_{\ell}) \, \zd \x+\tilde C_0},$ 
	where $\tilde C_0$ is a non-negative constant to keep $\tilde r(t)$ being real. Then the $L^2$ gradient flow, formulated within the SAV framework, is expressed as
	\begin{align}
		\begin{cases}\label{vector-valued ac}
			(u_{\ell})_t = \mathcal{G} \braB{-\epsilon^2 \Delta u_{\ell} + \frac{\tilde r}{\sqrt{E_1 + \tilde C_0}} \frac{\delta E_1}{\delta u_{\ell}}}, \\
			\tilde r_t = \frac{1}{2\sqrt{E_1 + \tilde C_0}} \sum_{\ell=1}^k \myinnerb{\frac{\delta E_1}{\delta u_{\ell}}, (u_{\ell})_t},
		\end{cases}
	\end{align}
	where $E_1=\sum_{\ell=1}^k \int_{\Omega} F(u_{\ell}) \, \zd \x$ and the modified energy is defined as
	\begin{equation}
		E[u_1, \cdots, u_k, \tilde r] = \sum_{\ell=1}^k \frac{\epsilon^2}{2} \mynormb{\nabla u_{\ell}}^2 + \tilde r^2 - \tilde{C}_0.\\
	\end{equation}
	Therefore, by following the frame in Section 2, the formulation of the diagonally IMEX RRK method of \eqref{vector-valued ac} is as follows:
	
	\begin{equation}\label{imex rrk-sav-vector}
		\begin{cases}
			U_{\ell,ni} = u_{\ell,\gamma}^n + \tau \sum\limits_{j=1}^i a_{ij} L_{\ell,nj} + \tau \sum\limits_{j=1}^{i-1} \bar a_{ij} N_{\ell,nj}, \quad \ell=1,\cdots,k,\\
			R_{ni} = r_{\gamma}^n + \tau \sum\limits_{j=1}^{i-1} \bar a_{ij}\tilde{N}_{nj}, \quad i=1,\cdots,s,\\
			u_{\ell,\gamma}^{n+1} = u_{\ell,\gamma}^n + \gamma_n\tau \sum\limits_{i=1}^s b_i L_{\ell,ni} + \gamma_n\tau \sum\limits_{i=1}^s \bar b_i N_{\ell,ni}, \quad \ell=1,\cdots,k,\\
			r_{\gamma}^{n+1} = r_{\gamma}^n + \gamma_n\tau \sum\limits_{i=1}^s \bar b_i\tilde{N}_{ni},
		\end{cases}
	\end{equation}
	where $L_{\ell,nj} = L_{\ell}(U_{\ell,nj})$, $N_{\ell,nj} = N_{\ell}(U_{1,nj},\cdots,U_{k,nj},R_{nj})$ and $\tilde N_{nj}=\tilde N(U_{1,nj},\cdots,U_{k,nj},R_{nj})$ for $j=1,\dots,s$. Let
	
	\begin{align}\label{denominator_ne_mu}
		\sum_{\ell=1}^k \dfrac{\epsilon^2}{2} \mynormB{\nabla \braB{\sum\limits_{i=1}^s  b_i L_{\ell,ni}+ \sum\limits_{i=1}^s \bar b_i N_{\ell,ni}}}^2 + \braB{\sum\limits_{i=1}^s \bar b_i\tilde{N}_{ni}}^2 \ne 0,
	\end{align}
	then $\gamma_n$ satisfies
	
	\begin{align}\label{Def_gamma_n_mu}
		\gamma_n =
		\begin{cases}
			\frac{\sum\limits_{i=1}^s \kbrabg{\sum\limits_{\ell=1}^k \epsilon^2 \myinnerB{u_{\ell,\gamma}^n-U_{\ell,ni}, \Delta\bra{b_i L_{\ell,ni}+ \bar b_i N_{\ell,ni}}} -2(r_{\gamma}^n-R_{ni}) \bar b_i\tilde{N}_{ni}}} {\tau \kbrabg{\sum\limits_{\ell=1}^k \frac{\epsilon^2}{2} \mynormB{\nabla \braB{\sum\limits_{i=1}^s  b_i L_{\ell,ni}+ \sum\limits_{i=1}^s \bar b_i N_{\ell,ni}}}^2 + \braB{\sum\limits_{i=1}^s \bar b_i\tilde{N}_{ni}}^2}}, \, &\text{if } \eqref{denominator_ne_mu} \text{ holds},\\
			1, \, &\text{else}.
		\end{cases}
	\end{align}
	
	One can easily verify that the results
	with respect to the energy dissipation law and convergence orders presented in Sections 2 and 3 are also hold for \eqref{imex rrk-sav-vector} and \eqref{Def_gamma_n_mu}. 
	To validate the proposition, we examine the vector-valued Allen-Cahn equations ($\mathcal{G}=-\mathcal{I}$) as a specific example.

	First of all, we investigate the relaxation coeffcient $\gamma_n$ by computing $\mynorm{\gamma _n-1}_{\ell^{\infty}}$ and $\mynorm{G_n(1)}_{\ell^{\infty}}$ for various time steps. Here, we set $F(u_{\ell})=\frac{1}{4}u_{\ell}^2(1-u_{\ell})^2$, $\ell =3$ (namely 3 phase variables), $\epsilon= 0.01$, $\Omega = (-0.5, 0.5)^2$ and 
	the initial conditions are provided as
	$$u_1(x, y, 0) = u_2(x, y, 0) = 0.5 \cos(\pi x) \cos(\pi y),\, u_3(x, y, 0) = 1 - u_1(x, y, 0) - u_2(x, y, 0).$$ The results are presented in Figures \ref{F:Ex3_Gamma} and \ref{F:Ex3_Sn}, demonstrating that the convergence orders match the theoretical results.
	Besides, we also illustrate the energy stability by using the IMEX RRK$(3,2)$ method. The tiem step size is $\tau=10^{-2}$ and the final time is $T_0=20$. Figure \ref{F:energy_vac} shows that the modified energy preserve dissipation property.
	
	\begin{figure}
		\includegraphics[width=0.88\textwidth]{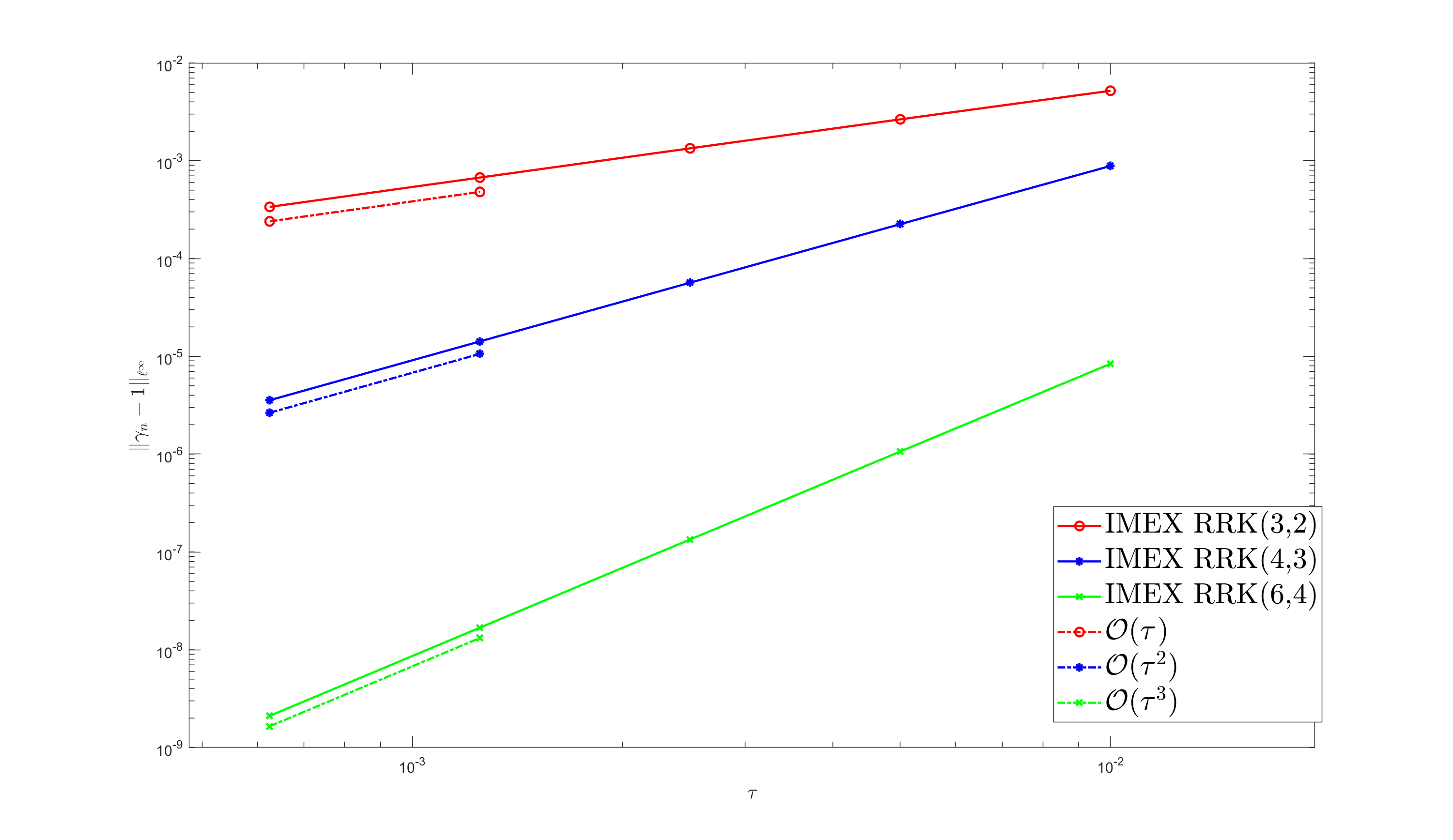}
		\caption{\label{F:Ex3_Gamma} $\mynorm{\gamma _n-1}_{\ell^{\infty}}$ for some relaxation (RT) methods (vector-valued AC equations)}
	\end{figure}
	
	\begin{figure}
		\includegraphics[width=0.88\textwidth]{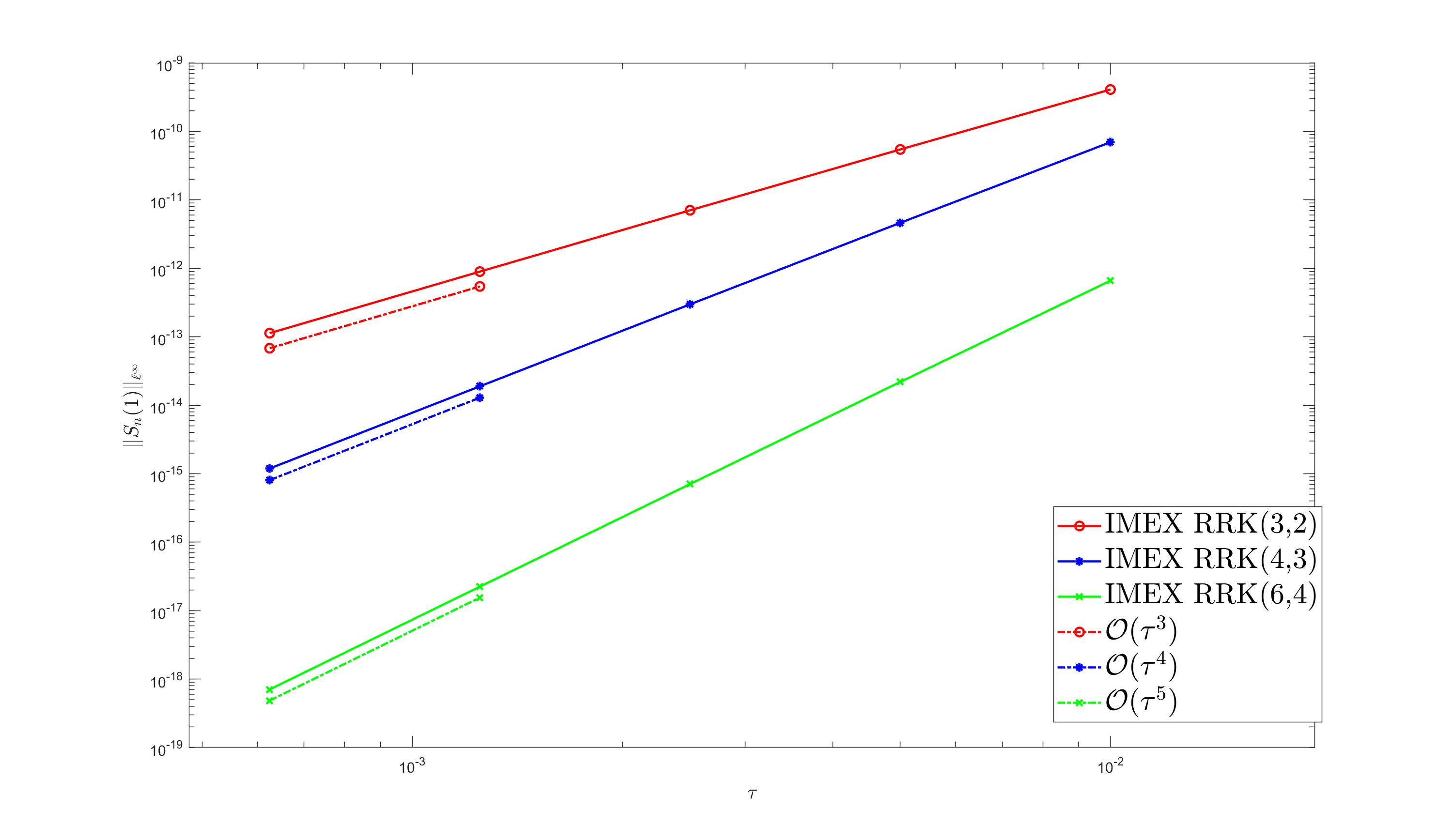}
		\caption{\label{F:Ex3_Sn} $\mynorm{G_n(1)}_{\ell^{\infty}}$ for some relaxation (RT) methods (vector-valued AC equations)}
	\end{figure}
	
	\begin{figure}
		\includegraphics[width=0.88\textwidth]{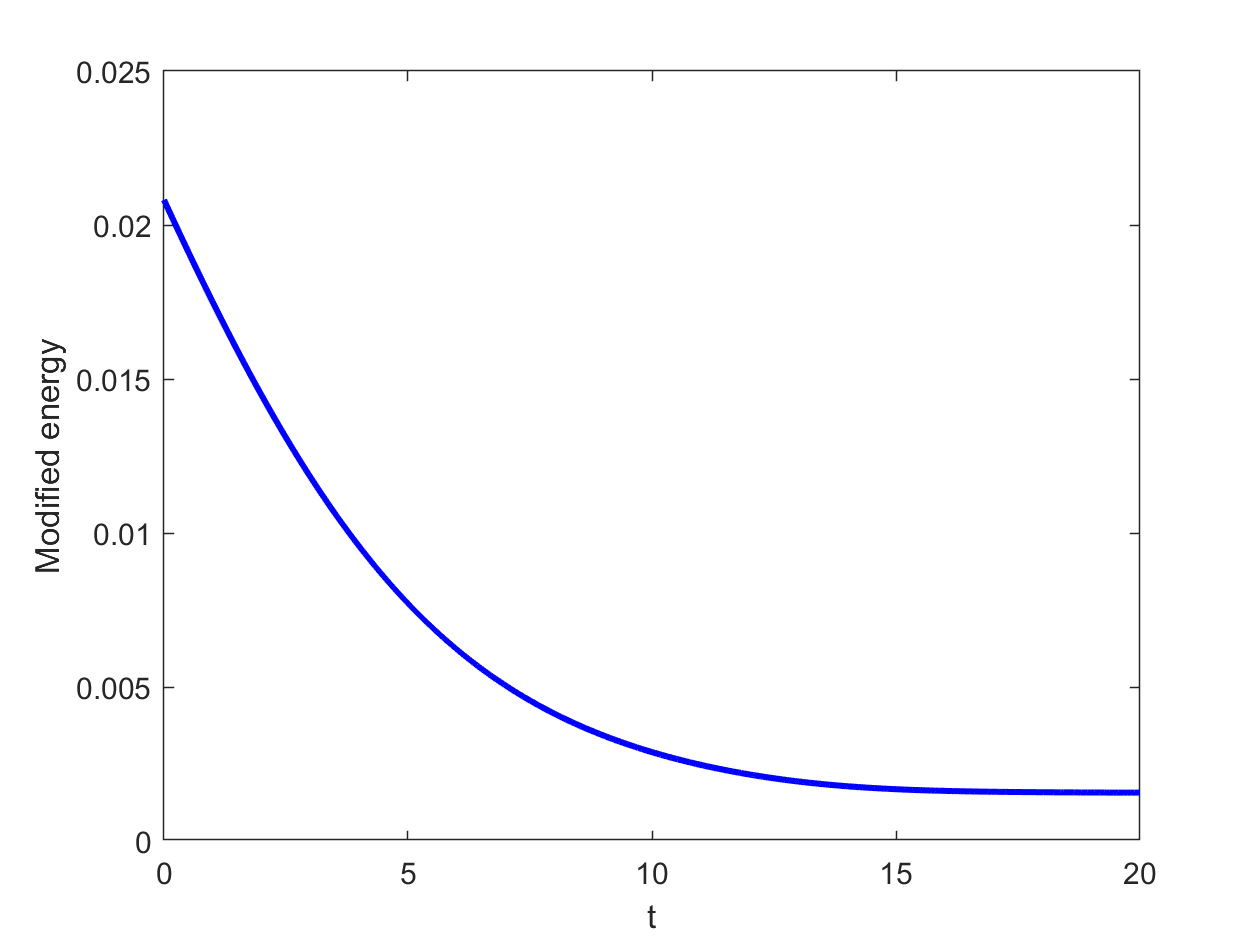}
		\caption{\label{F:energy_vac} Discrete energy evolution (vector-valued AC equations)} 
	\end{figure} 

	Next, we test the order of convergence at $T_0=1$ and use reference solutions computed with $\tau = 10^{-4}$. 
	Tables \ref{Table:EX3 ac-error} and \ref{Table:EX3 ac-error1} demonstrate that the results for all phase variables align with those in Examples 4.1 and 4.2, 
	namely the IDT methods have convergence orders $p-1$ while the RT methods achieve orders $p$,
	thereby reinforcing the analytical results.

	\begin{small}
		\begin{table}
			\caption{\label{Table:EX3 ac-error} $\ell^{\infty}$-norm errors and convergence orders for $u_1$/$u_2$ (vector-valued AC equations)}
			\begin{tabular}{|c|c|c|c|c|c|cc|cc|}
				\hline
				\multicolumn{2}{|c}{\multirow{2}*{$p$}} & \multicolumn{2}{|c}{\multirow{2}*{Method}} & \multicolumn{2}{|c|}{\multirow{2}*{$\tau$}} &  \multicolumn{2}{c|}{IDT} & \multicolumn{2}{c|}{RT}\\
				\multicolumn{2}{|c}{~} & \multicolumn{2}{|c}{~} & \multicolumn{2}{|c|}{~} & Error & Order & Error & Order\\
				\hline
				\multicolumn{2}{|c}{\multirow{4}*{2}} & \multicolumn{2}{|c}{\multirow{4}*{IMEX RRK(3,2)}} & \multicolumn{2}{|c|}{$1/10$} &  2.2141e-04 & - & 4.5259e-06  & - \\
				\multicolumn{2}{|c}{\multirow{4}*{~}} & \multicolumn{2}{|c}{\multirow{4}*{~}} & \multicolumn{2}{|c|}{$1/20$} &  1.1258e-04  & 0.9757 & 1.1108e-06 & 2.0267 \\
				\multicolumn{2}{|c}{\multirow{4}*{~}} & \multicolumn{2}{|c}{\multirow{4}*{~}} & \multicolumn{2}{|c|}{$1/40$} &  5.7390e-05  & 0.9721 & 2.7608e-07 & 2.0084 \\
				\multicolumn{2}{|c}{\multirow{4}*{~}} & \multicolumn{2}{|c}{\multirow{4}*{~}} & \multicolumn{2}{|c|}{$1/80$} &  2.8940e-05  & 0.9877 & 6.8843e-08 & 2.0037 \\
				\hline
				\multicolumn{2}{|c}{\multirow{4}*{3}} & \multicolumn{2}{|c}{\multirow{4}*{IMEX RRK(4,3)}} & \multicolumn{2}{|c|}{$1/10$} &  3.8443e-04  & - & 2.6736e-06 & - \\
				\multicolumn{2}{|c}{\multirow{4}*{~}} & \multicolumn{2}{|c}{\multirow{4}*{~}} & \multicolumn{2}{|c|}{$1/20$} &  8.9946e-05  & 2.0956 & 3.2987e-07  & 3.0188 \\
				\multicolumn{2}{|c}{\multirow{4}*{~}} & \multicolumn{2}{|c}{\multirow{4}*{~}} & \multicolumn{2}{|c|}{$1/40$} &  2.1791e-05  & 2.0454 & 4.0671e-08  & 3.0198 \\
				\multicolumn{2}{|c}{\multirow{4}*{~}} & \multicolumn{2}{|c}{\multirow{4}*{~}} & \multicolumn{2}{|c|}{$1/80$} &  5.3715e-06  & 2.0203 & 5.0451e-09  & 3.0111 \\
				\hline
				\multicolumn{2}{|c}{\multirow{4}*{4}} & \multicolumn{2}{|c}{\multirow{4}*{IMEX RRK(6,4)}} & \multicolumn{2}{|c|}{$1/10$} &  3.2952e-05  & - & 2.4396e-07  & - \\
				\multicolumn{2}{|c}{\multirow{4}*{~}} & \multicolumn{2}{|c}{\multirow{4}*{~}} & \multicolumn{2}{|c|}{$1/20$} &  3.8503e-06  & 3.0973 & 1.4365e-08  & 4.0861 \\
				\multicolumn{2}{|c}{\multirow{4}*{~}} & \multicolumn{2}{|c}{\multirow{4}*{~}} & \multicolumn{2}{|c|}{$1/40$} &  4.6280e-07  & 3.0565 & 8.6497e-10  & 4.0537 \\
				\multicolumn{2}{|c}{\multirow{4}*{~}} & \multicolumn{2}{|c}{\multirow{4}*{~}} & \multicolumn{2}{|c|}{$1/80$} &  5.6659e-08  & 3.0300 & 5.2984e-11  & 4.0290 \\
				\hline
			\end{tabular}
		\end{table}
	\end{small}

	\begin{small}
		\begin{table}
			\caption{\label{Table:EX3 ac-error1} $\ell^{\infty}$-norm errors and convergence orders for $u_3$ (vector-valued AC equations)}
			\begin{tabular}{|c|c|c|c|c|c|cc|cc|}
				\hline
				\multicolumn{2}{|c}{\multirow{2}*{$p$}} & \multicolumn{2}{|c}{\multirow{2}*{Method}} & \multicolumn{2}{|c|}{\multirow{2}*{$\tau$}} &  \multicolumn{2}{c|}{IDT} & \multicolumn{2}{c|}{RT}\\
				\multicolumn{2}{|c}{~} & \multicolumn{2}{|c}{~} & \multicolumn{2}{|c|}{~} & Error & Order & Error & Order\\
				\hline
				\multicolumn{2}{|c}{\multirow{4}*{2}} & \multicolumn{2}{|c}{\multirow{4}*{IMEX RRK(3,2)}} & \multicolumn{2}{|c|}{$1/10$} &  2.1948e-04 & - & 9.0905e-06  & - \\
				\multicolumn{2}{|c}{\multirow{4}*{~}} & \multicolumn{2}{|c}{\multirow{4}*{~}} & \multicolumn{2}{|c|}{$1/20$} &  1.1158e-04  & 0.9760 & 2.2375e-06 & 2.0225 \\
				\multicolumn{2}{|c}{\multirow{4}*{~}} & \multicolumn{2}{|c}{\multirow{4}*{~}} & \multicolumn{2}{|c|}{$1/40$} &  5.6870e-05  & 0.9723 & 5.5713e-07 & 2.0058 \\
				\multicolumn{2}{|c}{\multirow{4}*{~}} & \multicolumn{2}{|c}{\multirow{4}*{~}} & \multicolumn{2}{|c|}{$1/80$} &  2.8677e-05  & 0.9878 & 1.3907e-07 & 2.0022 \\
				\hline
				\multicolumn{2}{|c}{\multirow{4}*{3}} & \multicolumn{2}{|c}{\multirow{4}*{IMEX RRK(4,3)}} & \multicolumn{2}{|c|}{$1/10$} &  3.8108e-04  & - & 4.6089e-06 & - \\
				\multicolumn{2}{|c}{\multirow{4}*{~}} & \multicolumn{2}{|c}{\multirow{4}*{~}} & \multicolumn{2}{|c|}{$1/20$} &  8.9143e-05  & 2.0959 & 5.4450e-07  & 3.0814 \\
				\multicolumn{2}{|c}{\multirow{4}*{~}} & \multicolumn{2}{|c}{\multirow{4}*{~}} & \multicolumn{2}{|c|}{$1/40$} &  2.1593e-05  & 2.0455 & 6.5564e-08  & 3.0540 \\
				\multicolumn{2}{|c}{\multirow{4}*{~}} & \multicolumn{2}{|c}{\multirow{4}*{~}} & \multicolumn{2}{|c|}{$1/80$} &  5.3226e-06  & 2.0204 & 8.0352e-09  & 3.0285 \\
				\hline
				\multicolumn{2}{|c}{\multirow{4}*{4}} & \multicolumn{2}{|c}{\multirow{4}*{IMEX RRK(6,4)}} & \multicolumn{2}{|c|}{$1/10$} &  3.2664e-05  & - & 4.4150e-07  & - \\
				\multicolumn{2}{|c}{\multirow{4}*{~}} & \multicolumn{2}{|c}{\multirow{4}*{~}} & \multicolumn{2}{|c|}{$1/20$} &  3.8159e-06  & 3.0976 & 2.5151e-08  & 4.1337 \\
				\multicolumn{2}{|c}{\multirow{4}*{~}} & \multicolumn{2}{|c}{\multirow{4}*{~}} & \multicolumn{2}{|c|}{$1/40$} &  4.5861e-07  & 3.0567 & 1.4901e-09  & 4.0772 \\
				\multicolumn{2}{|c}{\multirow{4}*{~}} & \multicolumn{2}{|c}{\multirow{4}*{~}} & \multicolumn{2}{|c|}{$1/80$} &  5.6143e-08  & 3.0301 & 9.0550e-11  & 4.0405 \\
				\hline
			\end{tabular}
		\end{table}
	\end{small}
	
	Finally, we also simulate the evolution process using the IMEX RRK(3,2) method. In this case, we set the initial conditions
	\begin{align*}
		&u_{\ell}(x, y, 0) = \frac12 \braB{ 1 + \tanh \braB{\frac{r_1 - \sqrt{(x - x_{\ell})^2 + (y - y_{\ell})^2}}{\epsilon}}},\, \ell = 1, \,2,\\
		&u_3(x, y, 0) = 1 - u_1(x, y, 0) - u_2(x, y, 0),
	\end{align*}
	where $\epsilon=0.025$, $r_1=0.25$, $x_1 = 1.26$, $x_2 = 0.74$ and $y_1 = y_2 = 0.5$. Here, $u_1$ and $u_2$ are two closely circles with the same radius. Figure \ref{F:Ex3_evolution} shows snapshots of $u_1 + 2u_2$ at $t = 1, 10, 20, 40$ for $\tau=0.01$ and $\tau=0.1$. Both simulations show similar evolution with the circles merging together, which are agree with the findings of Example 2 in reference \cite{LiLi2024Unconditional}.
	
	\begin{figure}
		\includegraphics[width=0.24\textwidth]{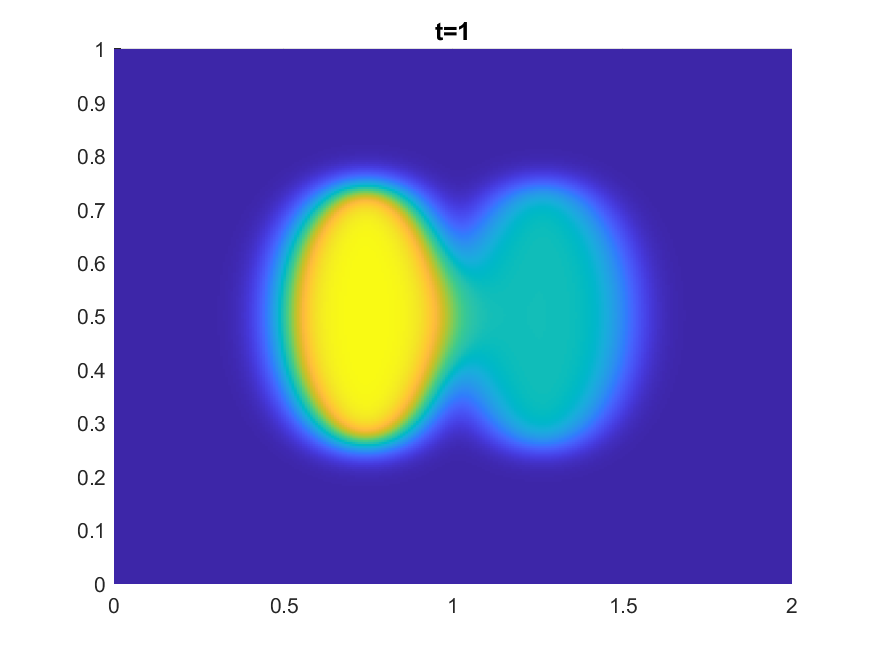}
		\includegraphics[width=0.24\textwidth]{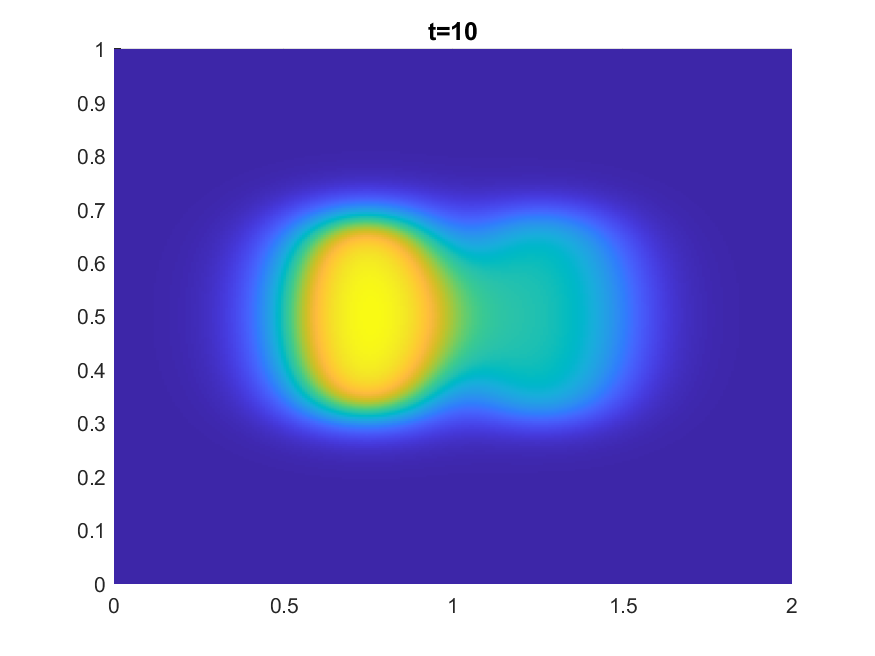}
		\includegraphics[width=0.24\textwidth]{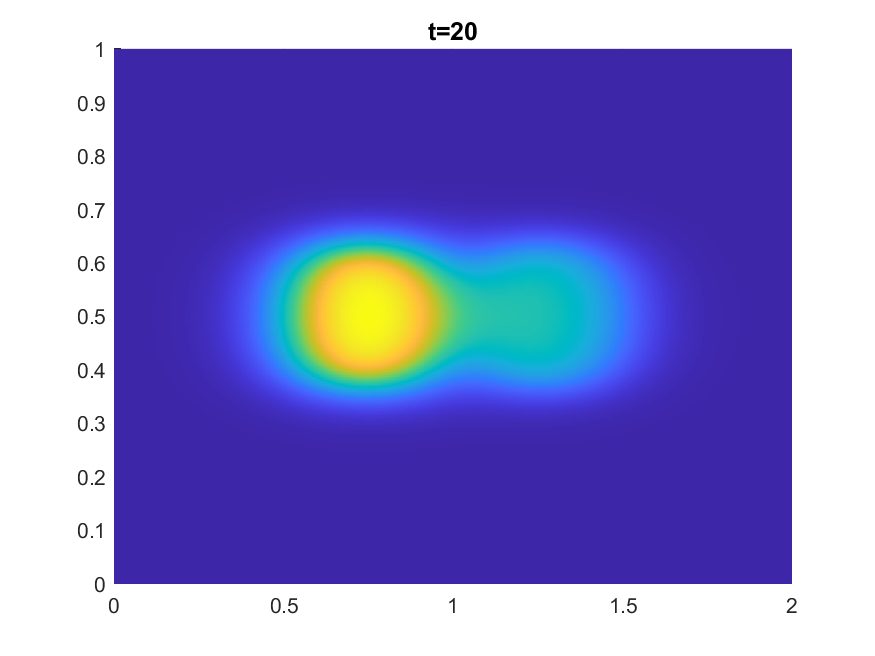}
		\includegraphics[width=0.24\textwidth]{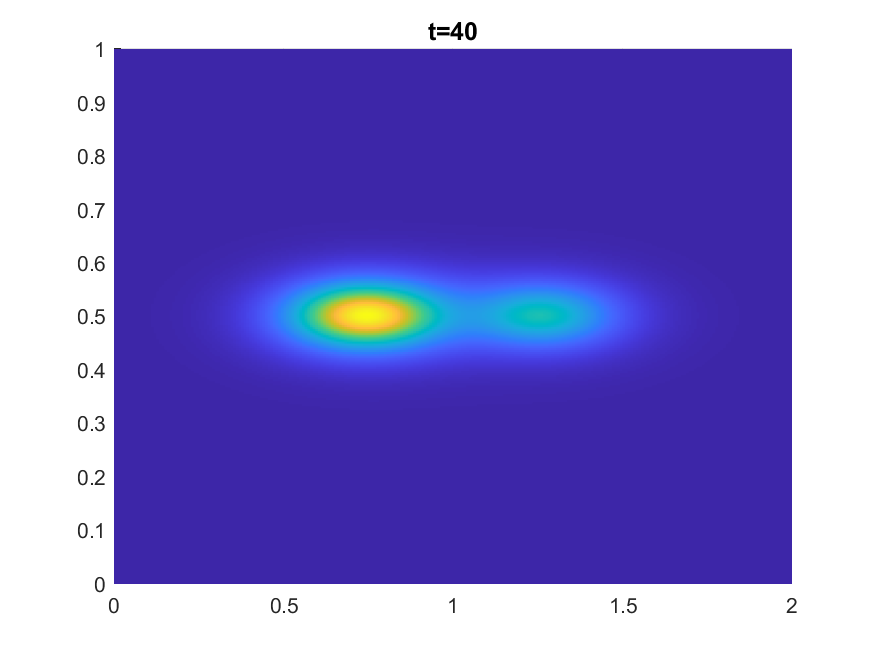}
		\includegraphics[width=0.24\textwidth]{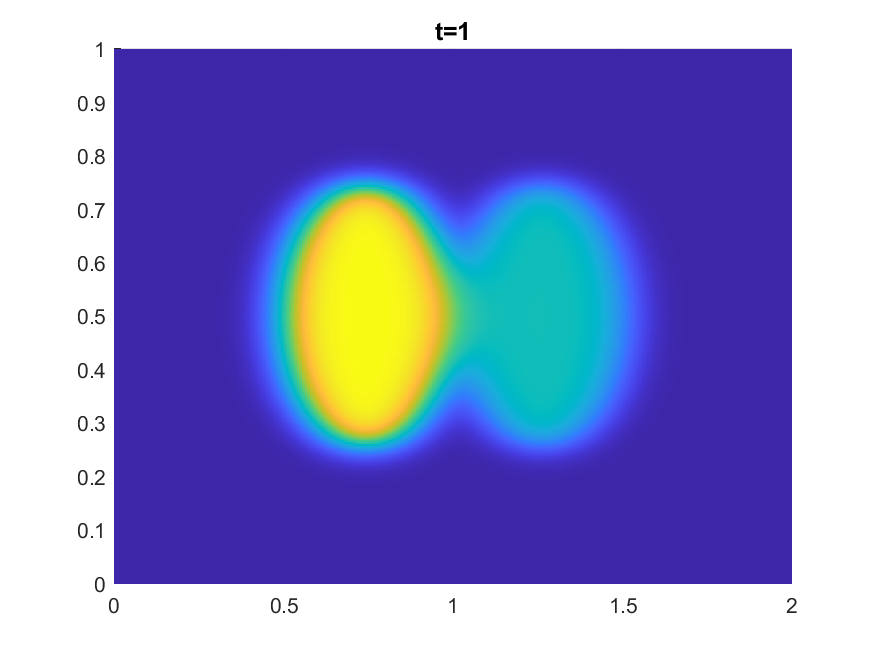}
		\includegraphics[width=0.24\textwidth]{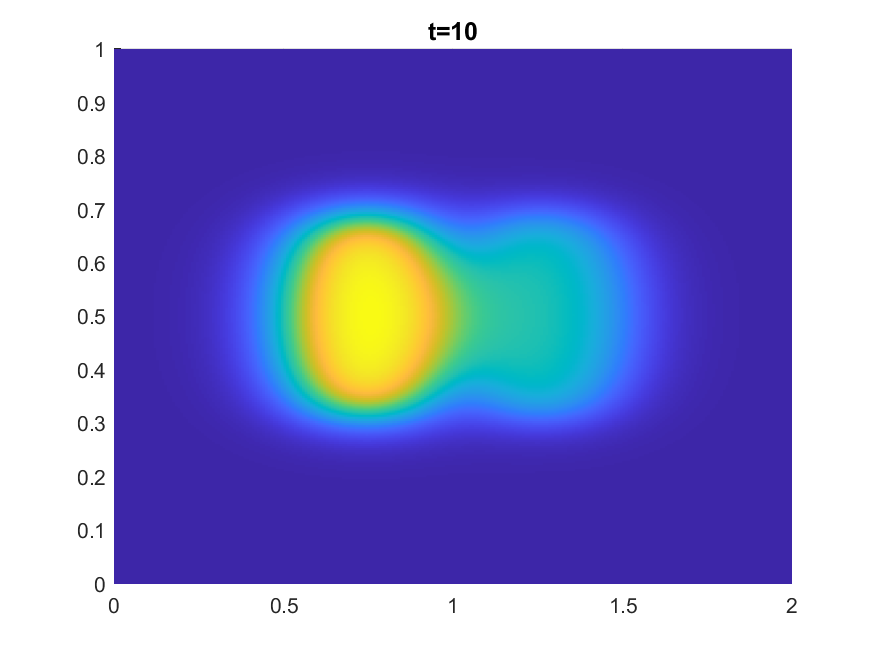}
		\includegraphics[width=0.24\textwidth]{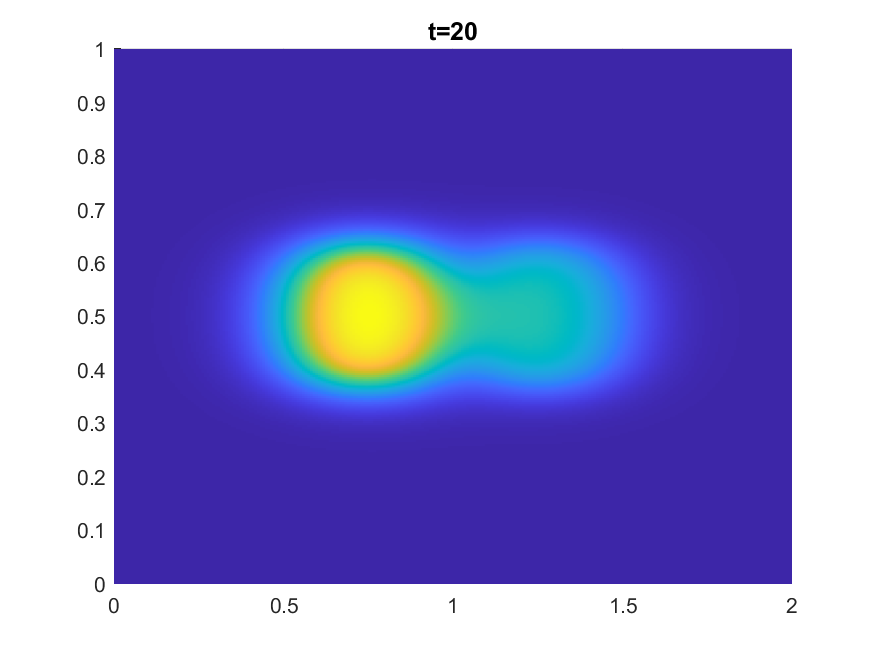}
		\includegraphics[width=0.24\textwidth]{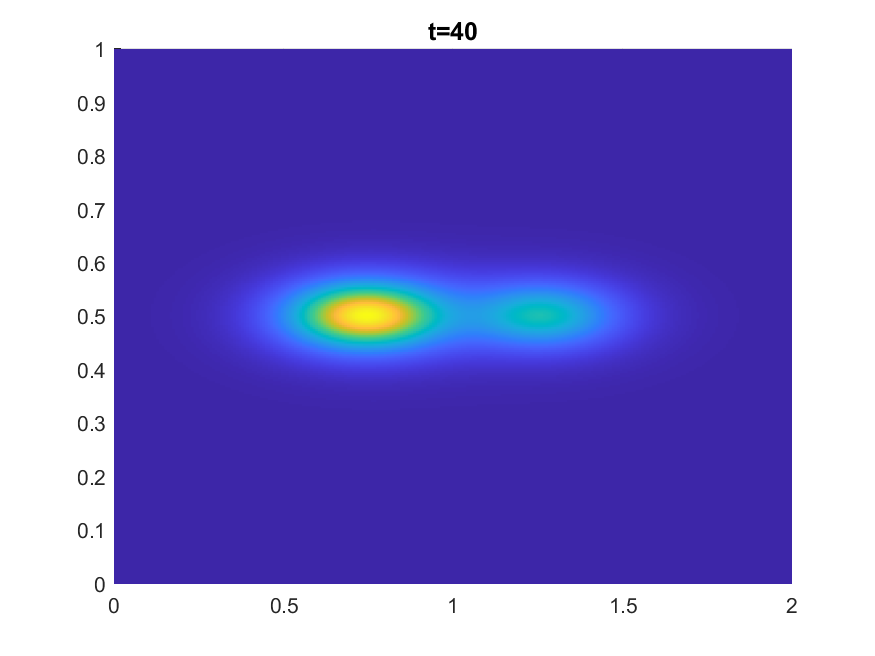}
		\caption{\label{F:Ex3_evolution} Vector-valued AC equations: Evolutions of the profile $u_1 + 2u_2$ at $t = 1, 10, 20, 40$ (from left to right) with different time step sizes; Top: $\tau = 0.01$; bottom: $\tau = 0.1$.}
	\end{figure}
	
	\section{Conclusions} 
	This paper is devoted to a class of IMEX RRK schemes based on the SAV method for solving the phase-field gradient flow models. Our core idea is to achieve high-order accuracy and balance accuracy with efficiency using the IMEX RK method, while combining the SAV method and the relaxation technique to stabilize nonlinear terms and guarantee unconditional energy dissipation. We rigorously prove that the scheme is unconditionally energy-stable and achieve the same order as the IMEX RK method. Numerical experiments validate our theoretical results. Furthermore, the scheme is extended to multi-component gradient flow. Using the vector-valued Allen-Cahn equations as an example, we demonstrate that the accuracy of the approximate solutions and the feasibility of phase separation evolution. In future work, we plan to extend our idea to other phase-field models, such as \cite{Wang2025GLOBAL,ZhangWang2024second}, thereby further broadening the scope of our research.

	\appendix
	\section*{Appendix A} 
	Based on references \cite{Kennedy2003Additive,Pareschi2005Implicit}, we present the coefficients of the IMEX RRK methods utilized in the numerical tests, noting that these coefficients are also applicable to the IMEX RK methods.
	
	\begin{itemize}
		\item IMEX RRK(3,2): 
		\[
		\qquad \qquad \qquad \qquad \text{Implicit method} \qquad \qquad \qquad \qquad \qquad \text{Explicit method} \qquad \qquad
		\]
		\[
		\renewcommand\arraystretch{1.2}
		\begin{array}{c|ccc}
			1-\sqrt{2}/2 & 1-\sqrt{2}/2 & 0 & 0 \\
			\sqrt{2}/2 & \sqrt{2}-1 & 1-\sqrt{2}/2 & 0 \\
			1/2 & \sqrt{2}/2-1/2 & 0 & 1-\sqrt{2}/2 \\
			\hline
			& 1/6 & 1/6 & 2/3
		\end{array}
		\qquad
		\renewcommand\arraystretch{1.2}
		\begin{array}{c|ccc}
			0 & 0 & 0 & 0 \\
			1 & 1 & 0 & 0 \\
			1/2 & 1/4 & 1/4 & 0 \\
			\hline
			& 1/6 & 1/6 & 2/3
		\end{array}
		\]
		\item IMEX RRK(4,3):
		\[
		\qquad \qquad \qquad \qquad \text{Implicit method} \qquad \qquad \qquad \qquad \qquad \text{Explicit method} \qquad \qquad
		\]	
		\[
		\renewcommand\arraystretch{1.2}
		\begin{array}
			{c|cccc}
			\alpha & \alpha & 0 & 0 & 0\\
			0 & -\alpha & \alpha & 0 & 0\\
			1 & 0 & 1-\alpha &  \alpha & 0\\
			1/2 & \beta & \eta &  1/2-\alpha-\beta-\eta & \alpha\\
			\hline
			& 0 & 1/6 & 1/6 & 2/3
		\end{array}	
		\qquad
		\renewcommand\arraystretch{1.2}
		\begin{array}
			{c|cccc}
			0 & 0 & 0 & 0 & 0\\
			0 & 0 & 0 & 0 & 0\\
			1 & 0 & 1 & 0 & 0\\
			1/2 & 0 & 1/4 & 1/4 &  0\\
			\hline
			& 0 & 1/6 & 1/6 & 2/3
		\end{array}	
		\]
		where the parameters $\alpha$, $\beta$ and $\eta$ take the values $0.24169426078821$, $0.06042356519705$ and $0.12915286960590$, respectively.	
		\item IMEX RRK(6,4):
		\[
		\qquad \qquad \qquad \qquad \qquad \text{Implicit method} \qquad \qquad \qquad \qquad \qquad
		\]
		\[
		\renewcommand\arraystretch{1.2}
		\begin{array}
			{c|cccccc}
			0 & 0 & 0 & 0 & 0 & 0 & 0\\
			\frac12 & \frac14 & \frac14 & 0 & 0 & 0 & 0\\
			\frac{83}{250} & \frac{8611}{62500} & \frac{1743}{31250} & \frac{1}{4} & 0 & 0 & 0\\
			\frac{31}{50} & \frac{5012029}{34652500} & \frac{-654441}{2922500} & \frac{174375}{388108} & \frac14 & 0 & 0\\
			\frac{17}{20} & \frac{15267082809}{155376265600} & \frac{-71443401}{120774400} & \frac{730878875}{902184768} & \frac{2285395}{8070912} & \frac14 & 0 \\
			1 & \frac{82889}{524892} & 0 & \frac{15625}{83664} & \frac{69875}{102672} & \frac{-2260}{8211} & \frac14 \\
			\hline
			& \frac{82889}{524892} & 0 & \frac{15625}{83664} & \frac{69875}{102672} & \frac{-2260}{8211} & \frac14
		\end{array}	
		\]
		\[
		\qquad \qquad \qquad \qquad \qquad \text{Explicit method} \qquad \qquad \qquad \qquad \qquad
		\]
		\[
		\renewcommand\arraystretch{1.2}
		\begin{array}
			{c|cccccc}
			0 & 0 & 0 & 0 & 0 & 0 & 0\\
			\frac12 & \frac12 & 0 & 0 & 0 & 0 & 0\\
			\frac{83}{250} & \frac{13861}{62500} & \frac{6889}{62500} & 0 & 0 & 0 & 0\\
			\frac{31}{50} & \frac{-116923316275}{2393684061468} & \frac{-2731218467317}{15368042101831} & \frac{9408046702089}{11113171139209} & 0 & 0 & 0\\
			\frac{17}{20} & \frac{-451086348788}{2902428689909} & \frac{-2682348792572}{7519795681897} & \frac{12662868775082}{11960479115383} & \frac{3355817975965}{11060851509271} & 0 & 0 \\
			1 & \frac{647845179188}{3216320057751} & \frac{73281519250}{8382639484533} & \frac{552539513391}{3454668386233} & \frac{3354512671639}{8306763924573} & \frac{4040}{17871} & 0 \\
			\hline
			& \frac{82889}{524892} & 0 & \frac{15625}{83664} & \frac{69875}{102672} & \frac{-2260}{8211} & \frac14
		\end{array}	
		\]
	\end{itemize}
	
	\bibliographystyle{abbrv}
	\bibliography{ref_imexrrk}
	
\end{document}